\documentclass[nolayout]{article}

\usepackage{authblk}
\usepackage[utf8]{inputenc}
\usepackage[english]{babel}
\usepackage{amsmath,amsthm}
\usepackage{geometry}
\usepackage{tikz}
\usetikzlibrary{fit,positioning}
\usepackage[textwidth=4cm,textsize=footnotesize]{todonotes}
\usepackage{fancyhdr}
\usepackage{amssymb,color,bbm,xargs}
\usepackage{graphicx}
\usepackage[active]{srcltx}
\usepackage{ifthen}
\usepackage{enumerate}
\usepackage{dsfont}
\usepackage{subcaption}
\usepackage{hyperref}
\usepackage{appendix}
\hypersetup{
    bookmarks=true,         
    colorlinks=true,       
    linkcolor=red,          
    citecolor=green,        
    filecolor=magenta,      
    urlcolor=cyan           
}

\newtheorem{lemma}{Lemma}
\newtheorem{proposition}[lemma]{Proposition}
\newtheorem{theorem}[lemma]{Theorem}

\newtheorem{Example}{Example}

\definecolor{lavander}{cmyk}{0,0.48,0,0}
\definecolor{violet}{cmyk}{0.79,0.88,0,0}
\definecolor{burntorange}{cmyk}{0,0.52,1,0}
\definecolor{burntgreen}{cmyk}{0.62,0.44,0.47,0}
\definecolor{burntblue}{cmyk}{0.86,0.30,0.18,0}
\definecolor{palegreen}{cmyk}{0.86,0.30,0.96,0}

\def\sup{\mathrm{sup}}

\def\1{\mathds{1}}
\def\rset{\mathbb{R}}
\def\rmd{\mathrm{d}}
\def\eqsp{\,}

\def\bayes{\pi_{\star}}

\newcommand{\limit}[1]{\underset{#1\to \infty}{\longrightarrow}}
\newcommand{\E}{\mathbb{E}}
\newcommand{\kullback}{\mathsf{L}}
\newcommand{\link}{\leftrightarrow}
\newcommand{\un}[1]{\mathds{1}_{#1}}
\newcommand{\nobs}{{\mathsf n}}

\newcommand{\pa}[1]{\left(#1\right)}
\newcommand{\cro}[1]{\left[#1\right]}
\newcommand{\absj}[1]{\left|#1\right|}
\newcommand{\norm}[1]{\left\|#1\right\|}
\newcommand{\set}[1]{\left\{#1\right\}}
\newcommand{\PE}[1]{\left\lfloor#1\right\rfloor}

\newcommand{\likelihood}[2]{\mathbb{P}^{#1}_{#2}}

\newcommand{\Lo}[2]{\ell^{#1}\pa{#2}}

\newcommand{\bN}{\mathbb{N}}
\newcommand{\bP}{\mathbb{P}}
\newcommand{\bbP}{\mathbf{P}}

\newcommand{\bZ}{\mathbb{Z}}

\newcommand{\cA}{\mathcal{A}}

\newcommand{\cE}{\mathcal{E}}
\newcommand{\cO}{\mathcal{O}}
\newcommand{\cP}{\mathcal{P}}

\newcommand{\cV}{\mathcal{V}}
\newcommand{\cX}{\mathcal{X}}
\newcommand{\cY}{\mathcal{Y}}

\newcommand{\Vset}{\mathbb{V}}
\newcommand{\Xset}{\mathbb{X}}

\newcommand{\ordermax}[2]{{\mathsf{q}}^{#1}_{#2}}
\newcommand{\remainder}[2]{{\mathsf{r}}^{#1}_{#2}}

\newcommand{\condlik}{k}

\DeclareMathOperator*{\argmax}{argmax}
\DeclareMathOperator*{\argmin}{argmin}

\newcommand{\shift}{\vartheta}
\newcommand{\MLE}{\widehat{\pi}}

\newcounter{hypH}
\newenvironment{hypH}{\refstepcounter{hypH}\begin{itemize}
\item[{\bf H\arabic{hypH}}]}{\end{itemize}}

\begin{document}

\title{Learning the distribution of latent variables in paired comparison models with round-robin scheduling}
\date{}

\author[*]{Roland Diel}
\author[$\dag$]{Sylvain Le Corff}
\author[$\ddag$]{Matthieu Lerasle}

\affil[*]{\small{Laboratoire J.A.Dieudonn\'e, UMR CNRS-UNS 6621, Universit\'e de Nice Sophia-Antipolis}}
\affil[$\dag$]{\small{Samovar, T\'el\'ecom SudParis, D\'epartement CITI, TIPIC, Institut Polytechnique de Paris}}
\affil[$\ddag$]{\small{CNRS, ENSAE, CREST, Institut Polytechnique de Paris}}

\lhead{R. Diel et al.}
\rhead{Learning the distribution of latent variables in paired comparison models}

\maketitle

\begin{abstract}
Paired comparison data considered in this paper originate from the comparison of a large number $N$ of individuals in couples. 
The dataset is a collection of results of contests between two individuals when each of them has faced $n$ opponents, where $n\ll N$.  Individual are represented by independent and identically distributed random parameters characterizing their abilities.
The paper studies the maximum likelihood estimator of the parameters distribution.  The analysis relies on the construction of a graphical model encoding conditional dependencies of the observations which are the outcomes of the first $n$ contests each individual is involved in. 
This graphical model allows to prove geometric loss of memory properties and deduce the asymptotic behavior of the likelihood function. This paper sets the focus on graphical models obtained from round-robin scheduling of these contests.
Following a classical construction in learning theory, the asymptotic likelihood is used to measure performance of the maximum likelihood estimator. 
Risk bounds for this estimator are finally obtained by sub-Gaussian deviation results for Markov chains applied to the graphical model.
\end{abstract}



\section{Introduction}
Consider a paired comparison problem involving a large number $N$ of individuals. 
For all $1\leqslant i\leqslant N$, the $i$-th individual is characterized by a {\em strength} (or {\em ability}) represented by an unknown parameter $V_i$. These parameters are indirectly observed through discrete valued scores $X_{i,j}$ describing the results of contests between individuals $i$ and $j$.
Given the values $V=(V_1,\ldots,V_N)$, the random variables $X_{i,j}$ are assumed to be independent and for each $i$ and $j$, the conditional distribution of $X_{i,j}$ given $V$ depends only on $V_i$ and $V_j$: there is a known function $k$ such that, for all $ 1\leqslant i<j\leqslant N$,
\[
\bP\pa{X_{i,j}=x|V}=k(x,V_i,V_j)\eqsp.
\]
The most classical example is the Bradley-Terry model \cite{ bradley:terry:1952,zemerlo:1929} where $x\in\{0,1\}$ and $k(1,V_i,V_j)=V_{i}/(V_i+V_j)$. 
In the seminal works \cite{ bradley:terry:1952,zemerlo:1929}, the problem was to recover the strengths $(V_1,\ldots,V_N)$ of a small number of players when the number of observed scores for each pair grows to infinity, see \cite{David:1963} for a review of these results in the original Bradley-Terry model and some of its extensions. 
More recently, \cite{Simons_Yao:1999} considered the problem of estimating each strength based on one score per pair in a tournament where the number $N$ of players grows to infinity. 
This framework led to several developments in computational statistics for the Bradley-Terry model, see \cite{Hun:2004} and \cite{caron:doucet:2012} for various extensions of this original model. 
The related Chen-Lu model was considered in \cite{chatterjee:diaconis:sly:2011} where the observations take values in $\{0,1\}$ and where the function $k$ is given by $k(1,V_i,V_j)=V_iV_j/(1+V_iV_j)$. Using one observation per pair of nodes, it is proved in \cite{chatterjee:diaconis:sly:2011} that, with probability asymptotically larger than $1-1/N^2$, there exists a unique maximum likelihood estimator of the nodes strengths which is such that the supremum norm of the estimation error is upper bounded by $\sqrt{\log N/N}$.

Consider the random oriented graph $G=(\{1,\ldots,N\},E)$, where an edge is drawn from $i$ to $j$ in $E$ if $X_{i,j}=1$ when $i<j$ and if $X_{j,i}=0$ when $i>j$. It is known since \cite{zemerlo:1929} that a necessary and sufficient condition for the existence of the maximum likelihood estimator (MLE) of $(V_1,\ldots,V_N)$ in the Bradley-Terry model is that $G$ is connected, i.e. there is a path between every pair of nodes.  This assumption implies some restrictions on the ratio between the strongest and the weakest strength \cite{Simons_Yao:1999}.  This prevents the use of maximum likelihood estimation in a sparse setting where the objective is to predict the outcome of future comparisons based on few observations.  This problem was for instance considered in \cite{YanYangXu:2011} which analyzes the MLE of $(V_1,\ldots,V_N)$ under the condition of existence of \cite{zemerlo:1929}, but in a graph where some edges may be unobserved.

This paper sets the focus on the case where each individual is compared to $n$ others, with possibly $n\ll N$ in such a way that the assumption of \cite{zemerlo:1929} may not hold. 
In other words, \emph{the MLE of $V_1,\ldots,V_N$ may not exist} in this setting.
To the best of our knowledge, this kind of dataset has not been analyzed previously and it is not clear what quantities can be recovered from these observations. Our strategy is motivated by the \textit{Bradley-Terry model in random environment} \cite{Sir_Red:2009,CheDielLer:2017}. 
In this model, strengths are supposed to be realizations of independent and identically distributed random variables with common distribution $\bayes$. 
The paper \cite{Sir_Red:2009} illustrated for example that an elementary parametric model for the strength can be used to make predictions regarding the teams scores at the end of baseball tournaments. 
The paper \cite{CheDielLer:2017} recently proved that the player with maximal strength ends the tournament with the highest degree in the graph $G$ if the tail of the nodes weights distribution is sufficiently convex. 

The take-home message is that \emph{the strengths distribution $\bayes$ is relevant to predict future outcomes} which motivates  the estimation of $\bayes$.
As every player is supposed to meet exactly $n$ opponents, the observed graph is naturally $n$ regular (every node has the same degree $n$). It is also assumed  that players meet according to the round-robin scheduling (see Section~\ref{sec:setting} for a description of this algorithm), a famous method to build  $n$-regular graphs recursively. The round-robin algorithm is routinely used for example to manage scheduling in chess, bridge, sport and online gaming tournaments. The MLE of $\bayes$ is analyzed based on the observation of the scores of every contest of the first $n$ rounds of the algorithm.  

First, a graphical model encoding conditional dependencies between strengths and scores is built.
This representation allows to approximate the likelihood function using a stationary hidden Markov model \cite{cappe:moulines:ryden:2005}. The asymptotic behavior of the normalized loglikelihood is analyzed using loss of memory properties of the hidden Markov process, following essentially the approach of \cite{DoucMoulines}. 
Then, following \cite{MR1641250}, the limit of the normalized loglikelihood is used to define a risk function, see Section~\ref{sec:ConvLikelihood} for details on this construction. This risk is then bounded from above for finite values of the number $N$ of nodes using concentration inequalities for Markov Chains \cite{dedecker:gouezel:2015}.
The excess risk scales as Dudley's entropy of the underlying statistical model normalized by a term of order $\sqrt{N}$ when $n$ is fixed and $N\to\infty$. 
From a learning perspective, Dudley's entropy bound is known to be suboptimal in general, it can be replaced by a majorizing measure bound \cite{Talagrand:2014}  since it derives from a sub-Gaussian concentration inequality for the increments of the underlying process, see~\eqref{eq:subGaussIncr}. 

More generally, the methodology introduced in this paper leads the way to various research perspectives in several fields. 
For example, identifiability of nonparametric hidden Markov models with finite state spaces was established recently along with the first convergence properties of estimators of the unknown distributions, see  \cite{decastro:gassiat:lacour:2016} for a penalized least-squares estimator of the emission densities,  \cite{decastro:gassiat:lecorff:2017,vernet:2015,vernet:2017} for consistent estimation of the posterior distributions of the states and posterior concentration rates for the parameters  or \cite{lehericy:2017} for order estimation.
However, very few theoretical results are available for the nonparametric estimation of general state spaces hidden Markov models. In computational statistics, Bayesian estimators of the strengths have been studied in Bradley-Terry models \cite{Hun:2004} and other extensions, see for example \cite{caron:doucet:2012}. In \cite{lecorff:lerasle:vernet:2018}, the unknown distribution of hidden variables is analyzed in a Bayesian framework and contraction rates of the posterior distribution are obtained using the concentration inequality established in this paper. Designing new algorithms to compute the MLE of the prior would then be of great interest to derive empirical Bayes estimators \cite{MR0084919, MR2724758}. 

The paper is organized as follows. 
Section~\ref{sec:setting} details the model, the maximum likelihood estimator of  the strengths distribution and the round-robin algorithm.
Section~\ref{sec:CondDep} presents preliminary results. The graphical model encoding conditional dependencies in round-robin graphs with latent variables is displayed, and the Markov chain associated with this representation is shown to be well approximated by a geometrically ergodic Markov chain. The main results are gathered in Section~\ref{sec:main}: convergence of the likelihood is established when the number $N$ of nodes grows to $+\infty$ and risk bounds for the MLE are provided. 
Finally, Appendices~\ref{sec:round:robin} to \ref{sec:MainProofs} are devoted to the proofs of these results. 

\section{Setting}
\label{sec:setting}
\subsection*{Graphs with latent variables}
Let $N$ be a positive integer, $E$ a set of couples $(i,j)$ with $1\leqslant i< j\leqslant N$ and $G=(\{1,\ldots,N\},E)$ the corresponding oriented graph. 
Let $V_1,\ldots,V_N$ denote independent and identically distributed (i.i.d.) random variables taking values in a measurable set $\cV$ with common \textit{unknown} distribution $\bayes$.  
For all $(i,j)\in E$, let $X_{i,j}$ denote a random variable taking values in a finite set $\cX$ such that, conditionally on $V=(V_1,\ldots,V_N)$, the random variables $(X_{i,j})_{(i,j)\in E}$ are independent with conditional distributions given by 
\begin{equation*}
\bP(X_{i,j}=x|V)=\condlik(x,V_i,V_j)\eqsp,
\end{equation*}
where $\condlik: \cX\times\cV\times\cV\to[0,1]$  is a known function.
In the following, the sets $\cX, \cV$ and the scores $(X_{i,j})_{(i,j)\in E}$  are available while the vector $V$ is unknown and the objective is to estimate the distribution $\bayes$. The following examples of triplets $(\cX, \cV, k)$ have been considered in the literature.

\begin{Example}[Bradley-Terry model \cite{bradley:terry:1952}]
\label{ex:bt}
In this example, $\cV=(0,\infty)$, $\cX=\{0,1\}$ and for all $x\in\cX$,
\[
k(x,V_i,V_j) = \left(\frac{V_i}{V_i+V_j}\right)^x\left(\frac{V_j}{V_i+V_j}\right)^{1-x}\eqsp.
\]
\end{Example}
\begin{Example}[Extensions of Bradley-Terry model \cite{caron:doucet:2012}] 
In the following examples, $\cV=(0,\infty)$.
\begin{enumerate}[-]
\item Let $\theta>0$ and $\cX=\{0,1\}$. In the Bradley-Terry model with home advantage, if $i$ is home, for all $x\in\cX$,
\[
k(x,V_i,V_j)= \left(\frac{\theta V_i}{\theta V_i+V_j}\right)^x\left(\frac{V_j}{\theta V_i+V_j}\right)^{1-x}\eqsp.
\]
\item In the Bradley-Terry model with ties \cite{rao:kupper:1967}, $\cX=\{-1,0,1\}$ and
\[
k(1,V_i,V_j)= \frac{V_i}{V_i+ \theta V_j}\quad\mbox{and}\quad k(0,V_i,V_j)= \frac{(\theta^2-1) V_i V_j}{\left(\theta V_i+V_j\right)\left( V_i+\theta V_j\right)}\eqsp.
\]
\end{enumerate}
\end{Example}

\begin{Example}[Graphon model]
The probability that two nodes $i$ and $j$ are connected in the graphon model (i.e.  $(i,j)\in E$) is the random variable $\mathsf{W}(V_i,V_j)$ with $\mathsf{W}: \cV\times\cV \to [0,1]$ and $\cV\subset \rset^+$. In the context of this paper, this boils down to choosing $\cX=\{0,1\}$ and setting by convention $X_{i,j}=0$ if and only if $(i,j)\notin E$ with 
\[
k(x,V_i,V_j) = \mathsf{W}(V_i,V_j)^x\left(1 - \mathsf{W}(V_i,V_j)\right)^{1-x}\eqsp.
\]
The problem in the graphon model  is to estimate the matrix of connection probabilities $(\mathsf{W}(V_i,V_j))_{1\leqslant i,j\leqslant N}$ using the observations of the adjacency matrix, and assuming that the distribution of $V_i$ is given.
 
In our setting, the aim is different, we try to estimate $\bayes$, the law of
the latent variables, from a partial observation $E$ of the adjacency matrix and with a known function $\mathsf{W}$.
\end{Example}

\begin{Example}[Chen-Lu model]
Consider a random graph where $E$ is such that an edge is drawn between node $i$ and node $j$ (i.e. $(i,j)\in E$) with probability $V_iV_j/(1+V_iV_j)$, with for all $1\leqslant k\leqslant N$, $V_k\in\cV=(0,\infty)$. In the context of this paper, this boils down to choosing $\cX=\{0,1\}$ and setting by convention $X_{i,j}=0$ if and only if $(i,j)\notin E$ with 
\[
k(x,V_i,V_j) = \left(\frac{V_iV_j}{1+V_iV_j}\right)^x\left(\frac{1}{1+V_iV_j}\right)^{1-x}\eqsp.
\]
\end{Example}
\subsection*{Maximum likelihood estimator}
The aim of this paper is to estimate the distribution $\bayes$ of the hidden variables $V=(V_1,\ldots,V_N)$ from the observations $X^E=(X_{i,j})_{(i,j)\in E}$. Let $\cA$ be a $\sigma$-field on $\cV$ and $\Pi$ be a set of probability measures on $(\cV,\cA)$. The statistical model is not assumed to be well specified i.e. $\Pi$ may not contain $\bayes$. For all $\pi\in\Pi$, the joint distribution of $(X^E,V)$ is given, for any $x^{E}\in \cX^{|E|}$ and all $A\in\cA^{\otimes N}$ by
\begin{equation}
\label{eq:jointlaw}
\likelihood{E}{\pi}(X^E=x^E,V\in A)=\int\un{A}(v)\prod_{(i,j) \in E}k(x^{E}_{i,j},v_{i},v_j)\pi^{\otimes N}(\rmd v)\eqsp,
\end{equation}
where $\un{A}$ is the indicator function of the set $A$. Using the convention $\log 0=-\infty$, the log-likelihood is given, for all $\pi\in\Pi$, by
\[
\Lo{E}{\pi} = \log\,\likelihood{E}{\pi}(X^E)\eqsp\qquad\text{where}\qquad  \likelihood{E}{\pi}(X^E)=\likelihood{E}{\pi}(X^E,V\in\cV^N)\eqsp.
\]
In this paper, $\bayes$ is estimated by the maximum likelihood estimator $\MLE^{E}$ defined as any maximizer of the log-likelihood:
\[
\MLE^{E}\in \argmax_{\pi\in\Pi}\{\Lo{E}{\pi}\}\eqsp.
\]

\subsection*{Round-robin (RR) Scheduling}\label{sec:RRscheduling}
Assume that  $N$ is an even integer. In the case of a round-robin scheduling, at $t=1$, $2i-1$ is paired with $2i$, for all $i\in[N/2]$, as in Figure~\ref{fig:robin:day1}. At $t=2$, the RR permutation $\cP_{\text{RR}}$ is performed: node $1$ is fixed $\cP_{\text{RR}}(1)=1$, $\cP_{\text{RR}}(2)=3$, each odd integer $2i-1<N-1$ satisfies $\cP_{\text{RR}}(2i-1)=2i+1$, $\cP_{\text{RR}}(N-1)=N$ and each even integer $2i>2$ satisfies $\cP_{\text{RR}}(2i)=2(i-1)$. This permutation is illustrated by the graphical representation given in Figure~\ref{fig:robin:move}. Then, the RR pairing is performed as in Figure~\ref{fig:robin:day2}. At each time $t> 2$, a RR permutation is performed as in Figure~\ref{fig:robin:move} and followed by a RR pairing. Let $n\geqslant 1$ denote an integer. The RR graph denoted by $E^{n,N}_{{\rm RR}}$ studied in detail in this paper contains all pairs collected in the first $n$ pairings of the RR algorithm. Note that $E^{N-1,N}_{{\rm RR}}$ is the complete graph and that we focus on situations where $n\ll N$.

\begin{figure}[h!]
\begin{subfigure}{\textwidth}
  \centering
  \begin{tikzpicture}
\tikzstyle{main}=[circle, minimum size = 10mm, thick, draw =black!80, node distance = 6mm]
\tikzstyle{connect}=[-latex, thick]
\tikzstyle{box}=[rectangle, draw=black!100]
\node[main, fill = black!10] (v1){$\scriptstyle 1$ };
\node[main] (v3) [right=of v1] {$\scriptstyle 3$ };
\node[main] (v5) [right=of v3] {$\scriptstyle 5$ };
\node[main] (v7) [right=of v5] {$\scriptstyle {{2i-1}}$};
\node[main] (v9) [right=of v7] {$\scriptstyle{{N-3}}$};
\node[main] (v11) [right=of v9] {$\scriptstyle{{N-1}}$};
\node[main] (v2) [below=of v1] {$\scriptstyle 2$};
\node[main] (v4) [below=of v3] {$\scriptstyle 4$};
\node[main] (v6) [below=of v5] {$\scriptstyle 6$};
\node[main] (v8) [below=of v7] {$\scriptstyle{{2i}}$};
\node[main] (v10) [below=of v9] {$\scriptstyle{{N-2}}$};
\node[main] (v12) [below=of v11] {$\scriptstyle{{N}}$};
\draw (v1) -- (v2);
\draw (v3) -- (v4);
\draw (v5) -- (v6);
\draw (v7) -- (v8);
\draw (v9) -- (v10);  
\draw (v11) -- (v12);     
\path (v5) -- node[auto=false]{\ldots}  (v7); 
\path (v7) -- node[auto=false]{\ldots}  (v9); 
\path (v6) -- node[auto=false]{\ldots}  (v8); 
\path (v8) -- node[auto=false]{\ldots}  (v10);    
\end{tikzpicture}
\caption{Round-robin pairing, step $1$.}
\label{fig:robin:day1}
\end{subfigure}

\vspace{.7cm}

\begin{subfigure}{\textwidth}
  \centering
  \begin{tikzpicture}
\tikzstyle{main}=[circle, minimum size = 10mm, thick, draw =black!80, node distance = 6mm]
\tikzstyle{connect}=[-latex, thick]
\tikzstyle{box}=[rectangle, draw=black!100]
\node[main, fill = black!10] (v1){$\scriptstyle 1$ };
\node[main] (v3) [right=of v1] {$\scriptstyle 3$ };
\node[main] (v5) [right=of v3] {$\scriptstyle 5$ };
\node[main] (v7) [right=of v5] {$\scriptstyle{{2i-1}}$};
\node[main] (v9) [right=of v7] {$\scriptstyle{{N-3}}$};
\node[main] (v11) [right=of v9] {$\scriptstyle{{N-1}}$};
\node[main] (v2) [below=of v1] {$\scriptstyle 2$};
\node[main] (v4) [below=of v3] {$\scriptstyle 4$};
\node[main] (v6) [below=of v5] {$\scriptstyle 6$};
\node[main] (v8) [below=of v7] {$\scriptstyle{{2i}}$};
\node[main] (v10) [below=of v9] {$\scriptstyle{{N-2}}$};
\node[main] (v12) [below=of v11] {$\scriptstyle{{N}}$};
\path (v2) edge [connect] (v3);   
\path (v4) edge [connect] (v2);  
\path (v3) edge [connect] (v5);  
\path (v6) edge [connect] (v4); 
\path (v9) edge [connect] (v11); 
\path (v11) edge [connect] (v12);   
\path (v12) edge [connect] (v10);   
\path (v5) -- node[auto=false]{\ldots}  (v7); 
\path (v7) -- node[auto=false]{\ldots}  (v9); 
\path (v6) -- node[auto=false]{\ldots}  (v8); 
\path (v8) -- node[auto=false]{\ldots}  (v10);    
\end{tikzpicture}
\caption{Round-robin permutation.}
\label{fig:robin:move}
\end{subfigure}

\vspace{.7cm}

\begin{subfigure}{\textwidth}
\centering
\begin{tikzpicture}
\tikzstyle{main}=[circle, minimum size = 10mm, thick, draw =black!80, node distance = 6mm]
\tikzstyle{connect}=[-latex, thick]
\tikzstyle{box}=[rectangle, draw=black!100]
\node[main, fill = black!10] (v1){$\scriptstyle 1$ };
\node[main] (v3) [right=of v1] {$\scriptstyle 2$ };
\node[main] (v5) [right=of v3] {$\scriptstyle 3$ };
\node[main] (v7) [right=of v5] {$\scriptstyle{{2i-1}}$};
\node[main] (v9) [right=of v7] {$\scriptstyle{{N-5}}$};
\node[main] (v11) [right=of v9] {$\scriptstyle{{N-3}}$};
\node[main] (v2) [below=of v1] {$\scriptstyle 4$};
\node[main] (v4) [below=of v3] {$\scriptstyle 6$};
\node[main] (v6) [below=of v5] {$\scriptstyle 8$};
\node[main] (v8) [below=of v7] {$\scriptstyle {{2i}}$};
\node[main] (v10) [below=of v9] {$\scriptstyle{{N}}$};
\node[main] (v12) [below=of v11] {$\scriptstyle{{N-1}}$};
\draw (v1) -- (v2);
\draw (v3) -- (v4);
\draw (v5) -- (v6);
\draw (v7) -- (v8);
\draw (v9) -- (v10);  
\draw (v11) -- (v12);     
\path (v5) -- node[auto=false]{\ldots}  (v7); 
\path (v7) -- node[auto=false]{\ldots}  (v9); 
\path (v6) -- node[auto=false]{\ldots}  (v8); 
\path (v8) -- node[auto=false]{\ldots}  (v10);    
\end{tikzpicture}
\caption{Round-robin pairing, step $2$.}
\label{fig:robin:day2}
\end{subfigure}
\caption{Round-robin algorithm.}
\end{figure}

\section{Conditional dependencies of round-robin graphs}
\label{sec:CondDep}
%

%
%
Let $d^{E}_0$ denote the graph distance in $(\{1,\ldots,N\},E)$, that is $d^{E}_0(i,j)$ is the minimal length of a path between nodes $i$ and $j$. Write $\{V_1,\ldots,V_N\}=\cup_{q=0}^{N}V^{E}_{q}$, where $V^{E}_{0}= \{V_1\}$ and, for any $q\geqslant 1$, $V^{E}_{q}$ is the set of $V_i$ such that $d^{E}_0(1,i)=q$. 
Let $\ordermax{}{E}+1$ denote the maximal distance between $1$ and $i\in\{1,\ldots,N\}$:
\[
\ordermax{}{E}+1 = \max_{1\leqslant i\leqslant  N}\,d^{E}_0(1,i)\eqsp.
\]
\begin{enumerate}[-]
\item For all $1\leqslant q \leqslant \ordermax{}E+1$, let 
\[
X^{E}_{q\link q}=\{X_{i,j} : (i,j) \text{ or }(j,i)\in E,\; i\in V_q^{E},\; j\in V_q^{E}\}\eqsp.
\]
The set $X^{E}_{q\link q}$ gathers all $X_{i,j}$ such that $i$ and $j$ satisfy $d^{E}_0(1,i) = d^{E}_0(1,j)=q$.
\item For all $0\leqslant q \leqslant \ordermax{}{E}$, let 
\[
X^{E}_{q\link q+1}=\{X_{i,j} : (i,j) \text{ or }(j,i)\in E,\; i\in V_q^{E},\; j\in V_{q+1}^{E}\}\eqsp.
\]
The set $X^{E}_{q\link q+1}$ gathers all $X_{i,j}$ such that $d^{E}_0(1,i) = q$ and $d^{E}_0(1,j)=q+1$.
\end{enumerate} 
Finally, for any $0\leqslant q \leqslant \ordermax{}{E}$, let 
\begin{align*}
X_{q}^{E}& = X_{q \link q+1}^{E}\cup X_{q+1\link q+1}^{E}\eqsp.
\end{align*}
Following \cite{lauritzen:1996}, the distribution $\bP^{E}_{\pi}$, given in \eqref{eq:jointlaw}, can be factorized with respect to an oriented acyclic graph where graph separations represent conditional independence. The factorization illustrates a global Markov property such that two sets of random variables $U_1$ and $U_2$ are independent given a third set $Z$ if $U_1$ and $U_2$ are d-separated by $Z$ in the oriented acyclic graph. The sets $U_1$ and $U_2$ are d-separated by $Z$ if every path from $U_1$ to $U_2$ is blocked by $Z$:
\begin{enumerate}[-]
\item the path contains a node in $Z$, and the edges of the path do not meet head-to-head at this node.
\item the path contains a node {\em not} in $Z$, none of its descendants are in $Z$, and the edges of the path {\em do} meet head-to-head at this node.
\end{enumerate}
Conditional dependencies described by $\bP^{E}_{\pi}$ can be represented in the graphical model of Figure~\ref{fig:generic:graphicalmodel}. 

\begin{figure}[h!]
\centering
\begin{tikzpicture}
\tikzstyle{player}=[circle, minimum size = 12mm, thick, draw =black!80, node distance = 8mm]
\tikzstyle{game}=[rectangle, minimum size = 9mm, thick, draw =black!80, node distance = 9mm]
\tikzstyle{connect}=[-latex, thick]
\tikzstyle{box}=[rectangle, draw=black!100]
\node[player] (v0){$V^{E}_{0}$};
\node[game, fill = black!10] (X01) [above right =of v0] {$X^{E}_{0}$};
\node[player] (v1) [below right=of X01] {$V^{E}_{1}$ };

\path (v0) edge [connect] (X01);  
\path (v1) edge [connect] (X01);  

\node[game, fill = black!10] (X12) [above right =of v1] {$X^{E}_{1}$};
\node[player] (v2) [below right=of X12] {$V^{E}_{2}$ };

\path (v1) edge [connect] (X12);  
\path (v2) edge [connect] (X12);  

\node[player] (vq) [right=of v2,xshift = .7cm] {$V^{E}_{\ordermax{}{E}}$ };

\path (v2) -- node[auto=false]{\ldots}  (vq);

\node[game, fill = black!10] (Xq1) [above right =of vq] {$X^{E}_{\ordermax{}{E}}$};
\node[player] (vq1) [below right=of Xq1] {$V^{E}_{{\small \ordermax{}{E}+1}}$ };

\path (vq1) edge [connect] (Xq1); 
\path (vq) edge [connect] (Xq1); 

\end{tikzpicture}
\caption{Graphical model of paired comparisons contests.}
\label{fig:generic:graphicalmodel}
\end{figure}

For instance, $V_1^E$ is independent of $V_2^E$ ($Z=\emptyset$) as every path between them goes through $X_1^E$, which is not in $Z$, with two edges meeting head-to-head at $X_1^E$. For all $0\le q \le \ordermax{}{E}$ any path between $X^{E}_{q}$ and other vertices except $V_q^{E}$ and $V_{q+1}^{E}$ goes through $V_{q}^{E}$ or $V_{q+1}^{E}$ which means that $X^{E}_{q}$ is independent of all other nodes given $V_q^{E}$ and $V_{q+1}^{E}$ ($Z = \{V_{q}^{E},V_{q+1}^{E}\}$ and no head-to-head edges). In particular, for all $0\leqslant q \leqslant \ordermax{}{E}$, and all $\pi\in\Pi$,
\[
\bP^{E}_{\pi}\left(X^{E}_{q}\middle| V,X^{E}_{0:q-1}\right) = \bP^{E}_{\pi}\left(X^{E}_{q}\middle| V^{E}_{q},V^{E}_{q+1}\right)=\prod_{(i,j) :X_{i,j}\in X^{E}_{q}}\condlik\left(X_{i,j},V_i,V_j\right)\eqsp.
\]



\begin{lemma}\label{lem:RR:stat} 
Let $N\geqslant n\geqslant 1$ and let $(\{1,\ldots,N\},E^{n,N}_{\text{RR}})$ denote the corresponding round-robin graph defined in Section~\ref{sec:RRscheduling}. Assume that $2\leqslant n<N/4$. Then, $\ordermax{}{E^{n,N}_{\text{RR}}}$ is the quotient of the Euclidean division of $N/2-1$ by $n-1$, that is 
\[
N/2-1= \ordermax{}{E^{n,N}_{\text{RR}}}(n-1)+\remainder{n}{N}\quad\text{with}\quad 0\le \remainder{n}{N}<n-1\eqsp.
\]
Moreover, $(V^{E^{n,N}_{\text{RR}}}_{q+1},X^{E^{n,N}_{\text{RR}}}_{q})_{2\leqslant q\leqslant \ordermax{}{E^{n,N}_{\text{RR}}}-1}$ is a stationary Markov chain such that for all $2\leqslant q\leqslant \ordermax{}{E^{n,N}_{\text{RR}}}-1$,
\[
|V^{E^{n,N}_{\text{RR}}}_{q}| = 2(n-1)\eqsp, \quad |X^{E^{n,N}_{\text{RR}}}_{q}| = n(n-1)\eqsp.
\]
\end{lemma}
\noindent Lemma~\ref{lem:RR:stat} is proved in Section~\ref{sec:round:robin}. It shows that RR graphs can be approximated by stationary hidden Markov models. When $E=E^{n,N}_{\text{RR}}$, by Lemma~\ref{lem:RR:stat}, 
the joint sequence $(V^{E}_{q+1},X^{E}_q)_{2\leqslant q \leqslant \ordermax{}{E}-1}$ is a stationary Markov chain which points toward the following decomposition of the likelihood. 
\begin{equation}
\label{eq:decomp:likelihood}
\log \likelihood{E}{\pi}\left(X^{E}\right) = \log \likelihood{E}{\pi}\left(X_{2:\ordermax{}{E}-1}^{E}\right) + \log \likelihood{E}{\pi}\left(X_{0}^{E},X_{1}^{E},X_{\ordermax{}{E}}^{E}\middle|X_{2:\ordermax{}{E}-1}^{E}\right)\eqsp.
\end{equation}
It is shown in Section~\ref{sec:main} that under a minoration condition on the kernel $k$, the last term in \eqref{eq:decomp:likelihood} is $o(\ordermax{}{E})$ when $N$ grows to infinity. This implies that the first term is the leading term in the analysis of the likelihood's asymptotic behavior. 
The uniform minoration condition of $k$ also ensures that the joint Markov chain $(V^E_{q+1},X^E_{q})_{q\geqslant 2 }$ is uniformly ergodic and admits the whole space $\Vset\times\Xset$ as small set
with stationary distribution on $\Vset\times\Xset$ given by $(A,x_0)\mapsto \int \un{A}(v_1)\pi_V(\rmd v_1) \pi_V(\rmd v_0)\condlik(x_0,v_0,v_1)$.
The joint stationary Markov chain $(V^E_{q+1},X^E_{q})_{q\geqslant 2 }$ may then be extended to a stationary process $({\bf X^{n}},{\bf V^{n}})$ indexed by $\mathbb{Z}$ with the same transition kernel. Hereafter, the distribution of this extended chain is denoted by $\bbP_{\pi}^n$.  

\section{Risk bounds for the MLE}
\label{sec:main}
Section~\ref{sec:ConvLikelihood} computes the limit likelihood function and shows why this limit defines a natural risk function to evaluate the MLE. Risk bounds for the MLE are obtained in Section~\ref{sec:RBMLE} using concentration inequalities for Markov chains.

\subsection{Asymptotic analysis of the likelihood}\label{sec:ConvLikelihood}
The problem being reduced to the analysis of the graphical model represented in Figure~\ref{fig:generic:graphicalmodel}, convergence results follow from geometrically decaying mixing rates of the conditional laws of the strengths ${V}^{E}_{k}$ given the observations. These rates are established under the following assumption. For any probability distribution $\pi$, denote by $\mathrm{supp}(\pi)$ the support of $\pi$.
\begin{hypH}
\label{assum:strongmix}
There exists $\varepsilon>0$ such that for all $x\in\cX, \pi\in\Pi\cup\{\bayes\}$ and $v_1,v_2 \in \mathrm{supp}(\pi)$, $\condlik(x,v_{1},v_2)\geqslant\varepsilon$.
\end{hypH}
%
Define also the shift operator $\shift$ on $(\cX^{n(n-1)})^{\mathbb{Z}}$ by $(\shift x)_k = x_{k+1}$ for all $k\in\mathbb{Z}$ and all $x\in (\cX^{n(n-1)})^{\mathbb{Z}}$. The following result establishes loss of memory properties of the extended hidden Markov chain $({\bf X^{n}},{\bf V^{n}})$ as well as the asymptotic behavior of the likelihood. This is the first main result of the paper.

\begin{theorem}\label{thm:ConvVrais}
Assume H\ref{assum:strongmix} holds. Then, for all $\mathsf{n}'>\mathsf{n}\geqslant q$ and all $p'<p<q$ in $\mathbb{Z}$,
\begin{align*}
\sup_{\pi\in\Pi}\left|\log \bbP^n_{\pi}\left({\bf X}^n_{q}\middle|{\bf X}^n_{q+1:\mathsf{n}}\right) - \log \bbP^n_{\pi}\left({\bf X}^n_{q}\middle|{\bf X}^n_{q+1:\mathsf{n}'}\right) \right|&\leqslant \varepsilon^{-n^2}\left(1-\varepsilon^{n^2}\right)^{\mathsf{n}-q-1} \eqsp,\\
\sup_{\pi\in\Pi}\left|\log \bbP^n_{\pi}\left({\bf X}^n_{q}\middle|{\bf X}^n_{p:q-1}\right) - \log \bbP^n_{\pi}\left({\bf X}^n_{q}\middle|{\bf X}^n_{p':q-1}\right) \right|&\leqslant \varepsilon^{-n^2}\left(1-\varepsilon^{n^2}\right)^{q-p} \eqsp.
\end{align*}
As a consequence, there exists a function $\ell_{\pi}^{n}$ such that for all $q$ in $\mathbb{Z}$,

\begin{equation}\label{eq:bound:loglik}
\sup_{\pi\in\Pi} \absj{\log \bbP^n_{\pi}\left({\bf X}^n_{q}\middle|{\bf X}^n_{q+1:\mathsf{n}}\right)-\ell_{\pi}^{n}(\shift^q{\bf X^{n}})}\limit{\mathsf{n}}0,\qquad \bbP_{\bayes}^n\mbox{-a.s}\eqsp. 
\end{equation}
Finally, when $E=E^{n,N}_{\text{RR}}$, for all $\pi\in\Pi$, $\bbP_{\bayes}^n$-a.s.  and in $\mathrm{L}^1(\bbP_{\bayes}^n)$,
\begin{equation}\label{eq:ApproxRisk}
\frac{1}{\ordermax{}{E}}\log \likelihood{E}{\pi}\left(X^{E}\right) \limit{N} \kullback^n_{\bayes}(\pi) = \E_{\bayes}^n\left[\ell_{\pi}^{n}({\bf X^{n}})\right]\eqsp. 
\end{equation}
\end{theorem}

Theorem~\ref{thm:ConvVrais} is proved in Section~\ref{ProofThm2}. It establishes convergence of the likelihood to the limit $\kullback^n_{\bayes}(\pi)$ when the number of nodes $N\to \infty$ while $n$ remains fixed. The rate of almost sure convergence $\ordermax{}{E}$ is proportional to $N$ in this case by Lemma~\ref{lem:RR:stat}. Eq~\eqref{eq:ApproxRisk} is the key to understand the definition of the risk function used in Section~\ref{sec:RBMLE}. 

Let $Y,Y_1,\ldots,Y_N$ denote i.i.d. observations in $\cY$, let $F$ denote a set of parameters, and let $\ell: F\times\cY\to\rset$ denote a loss function. The empirical risk minimizer is defined in this context by
\[
\hat f_N^{\text{ERM}}=\argmin_{f\in F}\sum_{i=1}^N\ell(f,Y_i)\eqsp.
\]
If $\E[\ell(f,Y_1)]<\infty$ for all $f\in F$, the performance of any $f\in F$ is measured by the \emph{excess risk}  \cite{MassartNedelec2006}
\[
R(f)=\E\cro{\ell(f,Y)}-\E\cro{\ell(f^*,Y)}\eqsp,
\]
where $Y$ is a copy of $Y_1$, independent of $Y_{1},\ldots,Y_N$ and $f^*$ is the minimizer of $\E[\ell(f,Y)]$ over $F$. Note that, when $\E[\ell(f,Y_1)]<\infty$ for all $f\in F$, the normalized empirical criterion satisfies almost surely,
\[
\frac1N\sum_{i=1}^N\ell(f,Y_i)\to \E[\ell(f,Y_1)]\eqsp.
\]
Therefore, following for instance \cite{MR1641250, MR0474638}, the excess risk $R(f)$ in learning theory is the difference between the asymptotic normalized empirical loss evaluated at $f$ and the minimizer of this quantity. 

In this paper, the MLE minimizes over $\pi\in\Pi$ the loglikelihood $-\log \likelihood{E}{\pi}\left(X^{E}\right)$. Using the identifications $\pi\sim f$, $\Pi\sim F$ and $-\log \likelihood{E}{\pi}\left(X^{E}\right)\sim \sum_{i=1}^N\ell(f,Y_i)$, Theorem~\ref{thm:ConvVrais} suggests to use  $-\kullback^n_{\bayes}(\pi)$ as a surrogate for $\E\cro{\ell(f,Y)}$. Therefore, define, for all  $\pi\in\Pi$,
\begin{equation}\label{def:ExcessRisk}
R^n_{\bayes}(\pi)=\kullback^{n}_{\bayes}(\bayes)-\kullback^{n}_{\bayes}(\pi)\eqsp. 
\end{equation}
By Proposition~\ref{prop:max:likelihood}, $\bayes$ is actually a minimizer of $-\kullback^{n}_{\bayes}(\pi)$ over $\Pi\cup\{\bayes\}$. 
Therefore, $R^n_{\bayes}$ is a natural extension of the excess risk associated with the likelihood function.
Notice here that the model is non identifiable. 
Clearly, the observed distribution is not changed if the distribution $\pi$ of $V$ is replaced by the distribution of $\varphi(V)$, for any mapping $\varphi:\cV\to \cV$ such that $k(x,\varphi(v_1),\varphi(v_2))=k(x,v_1,v_2)$ for any $x\in\cX$, and $v_1,v_2$ in $\cV$.
For example, in the Bradley-Terry model, for any $\lambda>0$, $k(x,\lambda v_1,\lambda v_2)=k(x,v_1,v_2)$ for any $x\in\cX$, and $v_1,v_2$ in $\cV$.
It is not easy however to describe precisely the class of transformations that would leave the observed distribution invariant in general, specially for a fixed $n$. 
This is why, in the following, we focus on bounding the risk $R^n_{\bayes}(\hat{\pi})$ of the estimator $\hat{\pi}$ rather than trying to bound a distance between $\pi^*$ and $\hat{\pi}$.

\subsection{Non asymptotic deviation bounds for the MLE}\label{sec:RBMLE}
The following theorem provides nonasymptotic deviation bounds for the excess risk of the MLE. This is the main result of this paper. Let $\|\cdot\|_{\mathsf{tv}}$ denote the total variation norm : for any signed measure $\pi$ on $\cV$,
\[
\|\pi\|_{\mathsf{tv}} = \mathrm{sup}\left\{\int \pi(\rmd v)f(v)\eqsp:\eqsp f \mathrm{\, bounded\, and\, measurable\, on\, }\cV\eqsp,\|f\|_{\infty}=1\right\}\eqsp.
\]
\begin{theorem}
\label{th:risk}
Assume H\ref{assum:strongmix} holds and $(\{1,\ldots,N\},E)$ is the round-robin graph (that is $E=E^{n,N}_{\text{RR}}$). For any probability measures $\pi$ and $\pi'$, define
 \begin{equation}
 \label{eq:def:d}
 d(\pi,\pi')=
\begin{cases}
 \|\pi-\pi'\|_{\mathsf{tv}}\log\pa{\frac1{\|\pi-\pi'\|_{\mathsf{tv}}}}&\text{if} \;\|\pi-\pi'\|_{\mathsf{tv}}< \mathrm{e}^{-1}\eqsp,\\
  \|\pi-\pi'\|_{\mathsf{tv}}&\text{if} \;\|\pi-\pi'\|_{\mathsf{tv}}\geqslant \mathrm{e}^{-1}\eqsp.
\end{cases}
\end{equation}
Let $\mathsf{N}(\Pi\cup\{\bayes\},d,\epsilon)$ be the minimal number of balls of $d$-radius $\epsilon$ necessary to cover $\Pi\cup\{\bayes\}$. Then, there exists $c>0$ such that, for any $t>0$ and any $n,N\geqslant 1$,
\[
\likelihood{E}{\bayes}\pa{R^{n}_{\bayes}(\MLE^{E})>\frac{cn\varepsilon^{-6n^2}}{\sqrt{N}}\cro{\int_0^{+\infty}\sqrt{\log \mathsf{N}(\Pi\cup\{\bayes\},d,\epsilon)}\rmd\epsilon+t}}\le \mathrm{e}^{-t^2}\eqsp.
\] 
\end{theorem}
Theorem~\ref{th:risk} is proved in Section~\ref{sec:proofs:risk}.
It provides the first non asymptotic risk bounds for any estimator of $\bayes$.
Besides, to the best of our knowledge, the ``sparse"  observation setting where each player only faces a few opponent has never been considered previously, neither in the Bradley-Terry model nor in any extensions.
Theorem~\ref{th:risk} demonstrates that the estimation of the distribution $\bayes$ of the parameters $V$ is fundamentally different from the problem of estimating $V$ that is usually considered, at least in Bradley-Terry models. 
While estimating nodes weights is possible under Zermelo's strong connectivity condition \cite{zemerlo:1929, Simons_Yao:1999, YanYangXu:2011}, the estimation of their distribution can be performed without such condition. 

The quasi-metric $d$ defined in \eqref{eq:def:d} used to measure the entropy of $\Pi$ is not intuitive. However, it is easy to check that $d(\pi,\pi')\lesssim_{\alpha} \norm{\pi-\pi'}_{\text{tv}}^{1-\alpha}$ for any $\alpha>0$. It follows that, for any class $\Pi$ with polynomial entropy for the total variation distance, that is such that $\mathsf{N}(\Pi\cup\{\bayes\},\norm{\cdot}_{\text{tv}},\epsilon)\lesssim \epsilon^D$ for small $\epsilon$, Dudley's entropy integral for $d$ satisfies 
\[
\int_0^{+\infty}\sqrt{\log \mathsf{N}(\Pi\cup\{\bayes\},d,\epsilon)}\rmd\epsilon\lesssim_{\alpha} \sqrt{D}\eqsp.
\]
Therefore, ``slow rates" of convergence are obtained for the MLE. The polynomial growth $\mathsf{N}(\Pi\cup\{\bayes\},\norm{\cdot}_{\text{tv}},\epsilon)\lesssim \epsilon^D$ is extremely standard, see \cite[p271--274]{MR1652247} for various examples where this assumption is satisfied and our result applies. 
On the other hand, ``fast" rates of convergence remain an open question. 
In particular, the margin condition \cite{MR1765618} required to prove such rates would hold if the total variation distance between strengths distributions was bounded from above by the excess risk derived from the asymptotic of the likelihood.

\appendix
\appendixpage

The remaining of the paper is devoted to the proof of the main results. Section~\ref{sec:round:robin} proves Lemma~\ref{lem:RR:stat}, describing precisely the structure of the graphical model given in Figure~\ref{fig:generic:graphicalmodel} in the case of a round-robin scheduling. Then, Section~\ref{sec:ProbTools} establishes central tools for the analysis of the likelihood of stationary processes whose conditional dependences are described by the graphical model in Figure~\ref{fig:generic:graphicalmodel}. These results are stated as independent lemmas as they might be of independent interest.  Proofs of the main theorems are finally gathered in Section~\ref{sec:MainProofs}.

\section{Proof of Lemma~1}
\label{sec:round:robin}
This section details the sets $V_q^{E}$ and $X_q^{E}$ for $0\le q\le \ordermax{}{E}+1$ when $E=E^{n,N}_{\text{RR}}$ (cf. Figures~\ref{fig:robin:day1}-\ref{fig:robin:day2}). In the following, notations  $i$ are identified with $V_i$ for all $1\le i \le N$, we also use $E=E^{n,N}_{\text{RR}}$ to shorten notations. Lemma~\ref{lem:RR:stat} follows directly from Lemmas~\ref{lem:ElementsAtDistanceq} and~\ref{lem:cardX} below. To prove these lemmas, consider the following notations.
\begin{gather*}
\cE=\{4x-1,4x : x\in[\PE{N/4}]\}\quad\mbox{and}\quad \cO=[N]\setminus\cE\eqsp.
\end{gather*}
The notation $\cE$ (resp $\cO$) comes from the fact that $\cE$ (resp $\cO$) contains all indices of the form $4x$ (resp. of the form $(2(2x+1))$) which are paired with $1$ after an \emph{even} (resp \emph{odd}) number $n\le N/4$ of permutations of the round-robin algorithm. For all $1\le q\le \ordermax{}{E}$, let
\[
V_{q,e}^{E}= V_{q}^{E}\cap \cE\quad\mbox{and}\quad V_{q,o}^{E}= V_{q}^{E}\cap \cO\eqsp.
\]

\begin{lemma}\label{lem:ElementsAtDistanceq} 
Let $n,N\ge1$ and $(\{1,\ldots,N\},E)$ be the round-robin graph ($E=E^{n,N}_{\text{RR}}$). Assume that $2\le n<N/4$ and let $N/2-1= \ordermax{}{E}(n-1)+\remainder{}{E}$ where $0\le \remainder{}{E}<n-1$. Then, 
\begin{equation}\label{def:V1}
V_{1}^{E}=\{V_{2x} : x=1,\ldots,n\}\eqsp, 
\end{equation}
and, for any $2\le q \le \ordermax{}{E}$, 
\begin{multline}\label{def:Vq}
V_{q}^{E}=\{V_{2x+1} : x\in[(q-2)(n-1)+1,(q-1)(n-1)]\}\\
\cup\{V_{2x} : x\in[2+(q-1)(n-1),1+q(n-1)]\}\eqsp.
\end{multline}
Furthermore,
\begin{multline}\label{def:Vq01'}
V_{\ordermax{}{E}+1}^{E}=\{V_{2x+1} : x\in[(\ordermax{}{E}-1)(n-1)+1,\ordermax{}{E}(n-1)+\remainder{}{E}]\}\\
\cup\{V_{2x} : x\in[2+\ordermax{}{E}(n-1),1+\remainder{}{E}+\ordermax{}{E}(n-1)]\}\eqsp.
\end{multline}
Therefore, $|V_{0}^{E}|=1$, $|V_{1}^{E}|=n$ and for all $2\le q\le\ordermax{}{E}$, $\absj{V_{q}^{E}}=2(n-1)$.
\end{lemma}
\begin{proof} To ease the reading of this proof, one can check its arguments on Figures~\ref{fig:groups:r0} and~\ref{fig:groups:r1} illustrating the case $n=3$.

\begin{figure}[h!]
\centering
\begin{tikzpicture}
\tikzstyle{main}=[circle, minimum size = 8mm, thick, draw =black!80, node distance = 4mm]
\tikzstyle{V6}=[draw =black!80,circle,minimum size = 8mm, thick,node distance = 4mm,palegreen,bottom color=palegreen!30,top color= white, text=violet]
\tikzstyle{V5}=[draw =black!80,circle,minimum size = 8mm, thick,node distance = 4mm,burntblue,bottom color=burntblue!30,top color= white, text=violet]
\tikzstyle{V4}=[draw =black!80,circle,minimum size = 8mm, thick,node distance = 4mm, text=black]
\tikzstyle{V3}=[draw =black!80,circle,minimum size = 8mm, thick,node distance = 4mm,lavander,bottom color=lavander!30,top color= white, text=violet]
\tikzstyle{V2}=[draw =black!80,circle,minimum size = 8mm, thick,node distance = 4mm,burntorange,bottom color=burntorange!30,top color= white, text=violet]
\tikzstyle{V1}=[draw =black!80,circle,minimum size = 8mm, thick,node distance = 4mm, fill = black!20, text=black]

\tikzstyle{legend_V1}=[node distance = 5mm,rectangle, rounded corners, thin, fill = black!20, draw, text=black,
                           minimum width=1.2cm, minimum height=0.7cm]
\tikzstyle{legend_V2}=[node distance = 8.8mm,rectangle, rounded corners, thin,bottom color=burntorange!30, ,top color= white,
                           burntorange, fill= white, draw, text=violet,
                           minimum width=1.2cm, minimum height=0.7cm]
\tikzstyle{legend_V3}=[node distance = 8.8mm,rectangle, rounded corners, thin,bottom color=lavander!30, ,top color= white,
                           lavander, fill= white, draw, text=violet,
                           minimum width=1.2cm, minimum height=0.7cm]      
\tikzstyle{legend_V4}=[node distance = 8.8mm,rectangle, rounded corners, thin, fill= white, draw, text=black,
                           minimum width=1.2cm, minimum height=0.7cm]                                                                                 
\tikzstyle{legend_V5}=[node distance = 8.8mm,rectangle, rounded corners, thin,bottom color=burntblue!30, ,top color= white,
                           burntblue, fill= white, draw, text=violet,
                           minimum width=1.2cm, minimum height=0.6cm]     
\tikzstyle{legend_V6}=[node distance = 4.8mm,rectangle, rounded corners, thin,bottom color=palegreen!30, ,top color= white,
                           palegreen, fill= white, draw, text=violet,
                           minimum width=1.2cm, minimum height=0.6cm]                                                             
\tikzstyle{connect}=[-latex, thick]
\tikzstyle{box}=[rectangle, draw=black!100]
\node[V1] (v1){$1$ };
\node[V4] (v3) [right=of v1] {$3$ };
\node[V4] (v5) [right=of v3] {$5$ };
\node[V1] (v7) [right=of v5] {$7$ };
\node[V1] (v9) [right=of v7] {$9$ };
\node[V1] (v2) [below=of v1] {$2$};
\node[V1] (v4) [below=of v3] {$4$};
\node[V1] (v6) [below=of v5] {$6$};
\node[V4] (v8) [below=of v7] {$8$};
\node[V4] (v10) [below=of v9] {$10$};
\node[V1] (v12) [right=of v10] {$12$};
\node[V1] (v14) [right=of v12] {$14$};
\node[V4] (v16) [right=of v14,xshift= .5cm] {${\scriptscriptstyle N-6}$};
\node[V4] (v18) [right=of v16] {${\scriptscriptstyle N-4}$};
\node[V1] (v20) [right=of v18] {${\scriptscriptstyle N-2}$};
\node[V1] (v22) [right=of v20] {${\scriptscriptstyle N}$};
\node[V1] (v11) [above=of v16] {${\scriptscriptstyle N-7}$};
\node[V1] (v13) [above=of v18] {${\scriptscriptstyle N-5}$};
\node[V4] (v15) [above=of v20] {${\scriptscriptstyle N-3}$};
\node[V4] (v17) [above=of v22] {${\scriptscriptstyle N-1}$};
   
\path (v14) -- node[auto=false]{\ldots}  (v16); 
\path (v9) -- node[auto=false]{\ldots}  (v11); 
 
\node[legend_V1] [below=of v4] {\small{\textsc{$V_1^{E}$}}};
\node[legend_V4] [below right=of v6,xshift= .3cm] {\small{\textsc{$V_2^{E}$}}};
\node[legend_V1] [below right=of v10,xshift= .6cm,yshift=-.25cm] {\small{\textsc{$V_3^{E}$}}};
\node[legend_V4] [below right=of v14,xshift= .8cm,yshift=.01cm] {\small{\textsc{$V_{\ordermax{}{E}-1}^{E}$}}};
\node[legend_V1] (leg) [below right=of v18,xshift= .7cm,yshift=-.25cm] {\small{\textsc{$V_{\ordermax{}{E}}^{E}$}}};
\node[legend_V4] [above right =of v15,xshift= -.9cm] {\small{\textsc{$V_{\ordermax{}{E}+1}^{E}$}}};
\end{tikzpicture}
\caption{Elements of $\cV^{E}$, case $n=3$, $\remainder{}{E} = 0$.}
\label{fig:groups:r0}
\end{figure}

\begin{figure}[h!]
\centering
\begin{tikzpicture}
\tikzstyle{main}=[circle, minimum size = 8mm, thick, draw =black!80, node distance = 4mm]
\tikzstyle{V6}=[draw =black!80,circle,minimum size = 8mm, thick,node distance = 4mm,palegreen,bottom color=palegreen!30,top color= white, text=violet]
\tikzstyle{V5}=[draw =black!80,circle,minimum size = 8mm, thick,node distance = 4mm,burntblue,bottom color=burntblue!30,top color= white, text=violet]
\tikzstyle{V4}=[draw =black!80,circle,minimum size = 8mm, thick,node distance = 4mm, text=black]
\tikzstyle{V3}=[draw =black!80,circle,minimum size = 8mm, thick,node distance = 4mm,lavander,bottom color=lavander!30,top color= white, text=violet]
\tikzstyle{V2}=[draw =black!80,circle,minimum size = 8mm, thick,node distance = 4mm,burntorange,bottom color=burntorange!30,top color= white, text=violet]
\tikzstyle{V1}=[draw =black!80,circle,minimum size = 8mm, thick,node distance = 4mm, fill = black!20, text=black]

\tikzstyle{legend_V1}=[node distance = 5mm,rectangle, rounded corners, thin, fill = black!20, draw, text=black,
                           minimum width=1.2cm, minimum height=0.7cm]
\tikzstyle{legend_V2}=[node distance = 8.8mm,rectangle, rounded corners, thin,bottom color=burntorange!30, ,top color= white,
                           burntorange, fill= white, draw, text=violet,
                           minimum width=1.2cm, minimum height=0.7cm]
\tikzstyle{legend_V3}=[node distance = 8.8mm,rectangle, rounded corners, thin,bottom color=lavander!30, ,top color= white,
                           lavander, fill= white, draw, text=violet,
                           minimum width=1.2cm, minimum height=0.7cm]      
\tikzstyle{legend_V4}=[node distance = 8.8mm,rectangle, rounded corners, thin, fill= white, draw, text=black,
                           minimum width=1.2cm, minimum height=0.7cm]                                                                                 
\tikzstyle{legend_V5}=[node distance = 8.8mm,rectangle, rounded corners, thin,bottom color=burntblue!30, ,top color= white,
                           burntblue, fill= white, draw, text=violet,
                           minimum width=1.2cm, minimum height=0.6cm]     
\tikzstyle{legend_V6}=[node distance = 4.8mm,rectangle, rounded corners, thin,bottom color=palegreen!30, ,top color= white,
                           palegreen, fill= white, draw, text=violet,
                           minimum width=1.2cm, minimum height=0.6cm]                                                             
\tikzstyle{connect}=[-latex, thick]
\tikzstyle{box}=[rectangle, draw=black!100]
\node[V1] (v1){$1$} ;
\node[V4] (v3) [right=of v1] {$3$};
\node[V4] (v5) [right=of v3] {$5$};
\node[V1] (v7) [right=of v5] {$7$};
\node[V1] (v9) [right=of v7] {$9$};
\node[V1] (v2) [below=of v1] {$2$};
\node[V1] (v4) [below=of v3] {$4$};
\node[V1] (v6) [below=of v5] {$6$};
\node[V4] (v8) [below=of v7] {$8$};
\node[V4] (v10) [below=of v9] {$10$};
\node[V1] (v12) [right=of v10] {$12$};
\node[V1] (v14) [right=of v12] {$14$};
\node[V4] (v24) [right=of v14,xshift= .5cm] {${\scriptscriptstyle N-8}$};
\node[V4] (v16) [right=of v24] {${\scriptscriptstyle N-6}$};
\node[V1] (v18) [right=of v16] {${\scriptscriptstyle N-4}$};
\node[V1] (v20) [right=of v18] {${\scriptscriptstyle N-2}$};
\node[V4] (v22) [right=of v20] {${\scriptscriptstyle N}$};
\node[V1] (v19) [above=of v24] {${\scriptscriptstyle N-9}$};
\node[V1] (v11) [above=of v16] {${\scriptscriptstyle N-7}$};
\node[V4] (v13) [above=of v18] {${\scriptscriptstyle N-5}$};
\node[V4] (v15) [above=of v20] {${\scriptscriptstyle N-3}$};
\node[V4] (v17) [above=of v22] {${\scriptscriptstyle N-1}$};
 
\path (v14) -- node[auto=false]{\ldots}  (v24); 
\path (v9) -- node[auto=false]{\ldots}  (v19); 

\node[legend_V1] [below=of v4] {\small{\textsc{$V_1^{E}$}}};
\node[legend_V4] [below right=of v6,xshift= .3cm] {\small{\textsc{$V_2^{E}$}}};
\node[legend_V1] [below right=of v10,xshift= .6cm,yshift=-.28cm] {\small{\textsc{$V_3^{E}$}}};
\node[legend_V4] [below right=of v14,xshift= .8cm,yshift=.01cm] {\small{\textsc{$V_{\ordermax{}{E}-1}^{E}$}}};
\node[legend_V1] (leg) [below right=of v18,xshift= -.65cm,yshift=-.25cm] {\small{\textsc{$V_{\ordermax{}{E}}^{E}$}}};
\node[legend_V4] [above right =of v15,xshift= -1.6cm] {\small{\textsc{$V_{\ordermax{}{E}+1}^{E}$}}};
\end{tikzpicture}
\caption{Elements of $\cV^{E}$, case $n=3,\ \remainder{}{E} = 1$.}
\label{fig:groups:r1}
\end{figure}

\noindent
We proceed by induction on $q$.
The definition of $V_{1}^{E}$ given by \eqref{def:V1} is straightforward. Then, $V_{2}^{E}$ contains:
\begin{enumerate}[-]
\item all $V_i$ paired with some $V_j\in V_{1}^{E}$ before the first RR permutation besides $V_1$ that does not belong to $V_{2}^{E}$. These are all  $\{V_{2x+1} : x=1,\ldots,n-1\}$ ;
\item all $V_i$ paired with $V_2$ and $V_4$ that are not in $V_{0}^{E}\cup V_{1}^{E}$. After $n$  RR permutations, all $V_i$ paired with $V_2$ are $\{V_1,V_{4x+2} : x=1,\ldots,n-1\}$ and those with $V_4$ are $\{V_1,V_3,V_{4x} : x=2,\ldots,n-2\}$. 
\end{enumerate}
Therefore, 
\[
V_{2}^{E}\supset \{V_{2x+1} : x=1,\ldots,n-1\}\cup\{V_{2x} : x=n+1,\ldots,2n-1\}\eqsp.
\]
On the other hand, by induction,  for all $i\notin \{N-2x+1, x=1,\ldots,2(n-1)\}\cup \{2x : x=1,\ldots, 2n-1\}$,
\begin{gather}
\notag\text{if }i \text{ is odd, it is paired with }\{V_{i+4x+1} : x=0,\ldots n-1\}\eqsp,\\
\label{eq:Opponents}\text{if }i \text{ is even, it is paired with }\{V_{i-4x-1} : x=0,\ldots,n-1\}\eqsp.
\end{gather}
This implies that there is no even number $i\ge 4n$ nor odd number $i>2n-1$ such that $V_i\in V_{2}^{n,N}$, which yields:
\begin{equation*}
 V_{2}^{E}= \{V_{2x+1} : x=1,\ldots,n-1\}\cup\{V_{2x} : x=n+1,\ldots,2n-1\}\eqsp.
\end{equation*}
\eqref{def:Vq} is obtained by induction using the same arguments and \eqref{def:Vq01'} is a direct consequence of the round-robin algorithm. 
The last claim follows by noting that for all $q\in[2,\ordermax{}{E}]$,
\[
|V_{q,e}^{E}|=|V_{q,o}^{E}|=n-1\eqsp.
\]
Indeed, one of the following cases holds.
\begin{enumerate}[-]
\item $n-1=2p$ for some $p\in \bN$. In this case, 
\[
|\{j: V_j\in V_{q,e}^{E},j\in 2\bZ\}|=|\{i : V_i\in V_{q,e}^{E},i\in 2\bZ+1\}|=p\eqsp.
\]
\item $n-1=2p+1$ for some $p\in \bN$. In this case, either 
\[
|\{j: V_j\in V_{q,e}^{E},j\in 2\bZ\}|=p,\qquad\text{and}\qquad |\{i : V_i\in V_{q,e}^{E},i\in 2\bZ+1\}|=p+1\eqsp,
\]
or 
\[
|\{j: V_j\in V_{q,e}^{E},j\in 2\bZ\}|=p+1,\qquad\text{and}\qquad |\{i : V_i\in V_{q,e}^{E},i\in 2\bZ+1\}|=p\eqsp.
\]
\end{enumerate}
\end{proof}
%
%

\begin{lemma}
\label{lem:cardX}
Let $n,N\ge1$ and $(\{1,\ldots,N\},E)$ be the round-robin graph ($E=E^{n,N}_{\text{RR}}$).  Then, for all $2\le q\le \ordermax{}{E}-1$, 
\[
|X_{q}^{E}| = n(n-1)\eqsp.
\]
\end{lemma}
\begin{proof}
The proof essentially consists in building the graphical model of Figure~\ref{fig:robin:graphicalmodel} from  the one displayed in Figure~\ref{fig:generic:graphicalmodel}.

\begin{figure}[h!]
\centering
\begin{tikzpicture}
\tikzstyle{player}=[circle, minimum size = 9mm, thick, draw =black!80, node distance = 8mm]
\tikzstyle{game}=[rectangle, minimum size = 9mm, thick, draw =black!80, node distance = 9mm]
\tikzstyle{connect}=[-latex, thick]
\tikzstyle{box}=[rectangle, draw=black!100]
\node[game, fill = black!10] (v1){$X^{E}_{0\link1,e}$};
\node[player] (v2) [below=of v1] {$V^{E}_{0}$};
\node[game, fill = black!10] (v3) [below=of v2] {$X^{E}_{0\link1,o}$};
\node[player] (v4) [right=of v1] {$V^{E}_{1,e}$ };
\node[player] (v5) [right=of v3] {$V^{E}_{1,o}$ };
\node[game, fill = black!10] (v6) [above=of v4] {$X^{E}_{1\link1,e}$};
\node[game, fill = black!10] (v7) [below=of v5] {$X^{E}_{1\link1,o}$};

\path (v2) edge [connect] (v3);  
\path (v2) edge [connect] (v1);  
\path (v4) edge [connect] (v1);  
\path (v4) edge [connect] (v6);  
\path (v5) edge [connect] (v3);  
\path (v5) edge [connect] (v7);  

\node[game, fill = black!10] (v8) [below right =of v4,yshift=+.2cm] {$X^{E}_{1\link 2,e}$};
\node[game, fill = black!10] (v9) [above right =of v5,yshift=-.2cm] {$X^{E}_{1\link 2,o}$};
\node[player] (v10) [above right=of v8,yshift=-.2cm] {$V^{E}_{2,e}$};
\node[player] (v11) [below right=of v9,yshift=+.2cm] {$V^{E}_{2,o}$};
\node[game, fill = black!10] (v12) [above=of v10] {$X^{E}_{2\link 2,e}$};
\node[game, fill = black!10] (v13) [below=of v11] {$X^{E}_{2\link2,o}$};

\path (v10) edge [connect] (v12);  
\path (v10) edge [connect] (v8);  
\path (v11) edge [connect] (v13);  
\path (v11) edge [connect] (v9);  
\path (v4) edge [connect] (v8);  
\path (v5) edge [connect] (v9);  

\node[player] (v14) [right=of v10,xshift=.7cm] {$V^{E}_{\ordermax{}{E},e}$};
\node[player] (v15) [right=of v11,xshift=.7cm] {$V^{E}_{\ordermax{}{E},o}$};

\path (v10) -- node[auto=false]{\ldots}  (v14); 
\path (v11) -- node[auto=false]{\ldots}  (v15); 

\node[game, fill = black!10] (v16) [above=of v14] {$X^{E}_{\ordermax{}{E}\link\ordermax{}{E},e}$};
\node[game, fill = black!10] (v17) [below=of v15] {$X^{E}_{\ordermax{}{E}\link\ordermax{}{E},o}$};

\path (v14) edge [connect] (v16);  
\path (v15) edge [connect] (v17);  

\node[game, fill = black!10] (v18) [below right=of v14,yshift=.3cm] {$X^{E}_{\ordermax{}{E}\link\ordermax{}{E}+1,e}$};
\node[game, fill = black!10] (v19) [above right=of v15,yshift=-.3cm] {$X^{E}_{\ordermax{}{E}\link\ordermax{}{E}+1,o}$};

\path (v14) edge [connect] (v18);  
\path (v15) edge [connect] (v19);  

\node[player] (v20) [right=of v19,yshift=.6cm] {$V^{E}_{\ordermax{}{E}+1}$};

\path (v20) edge [connect] (v18);  
\path (v20) edge [connect] (v19); 

\node[game, fill = black!10] (v21) [above=of v20] {$X^{E}_{\ordermax{}{E}+1\link\ordermax{}{E}+1}$};

\path (v20) edge [connect] (v21); 
  
\end{tikzpicture}
\caption{Graphical model of the round-robin algorithm.}
\label{fig:robin:graphicalmodel}
\end{figure}
\noindent Edges involving the first node are decomposed as:
\[
X_{0\link1,e}^{E}=\{X_{1,4x} : x=1,\ldots,\PE{n/2}\}=\{X_{1,i} : V_i\in V_{1,e}^{E}\}\quad\mbox{and}\quad X_{0\link 1,o}^{E}=\{X_{1,i} : V_i\in V_{1,o}^{E}\}\eqsp.
\]
Edges involving nodes in $V_{1}^{E}$ that are both different from $1$ are described as follows.
\begin{enumerate}[-]
\item  Edges between two nodes in $V_{1}^{E}$ denoted by: 
\begin{align*}
X_{1\link1,e}^{E}&=\{X_{4x,4y} : (x,y)\in [\PE{n/2}], x< y\}=\{X_{i,j} : V_i,V_j\in V_{1,e}^{E},i<j\}\eqsp,\\
X_{1\link1,o}^{E}&=\{X_{i,j} : V_i,V_j\in V_{1,o}^{E},i<j\}\eqsp. 
\end{align*}
Note that there is no edge between any $V_i\in V_{1,e}^{E}$ and a node $V_j\in V_{q,o}^{E}$ for any $q\ge 1$. In particular, there is no edge between any $V_i\in V_{1,e}^{E}$ and  $V_j\in V_{1,o}^{E}$.  Therefore, $X_{1\link1,e}^{E}\cup X_{1\link1,o}^{E}$ describes all edges between nodes in $V_{1}^{E}$. 
\item Edges between $V_i\in V_{1}^{E}$ and $V_j\in V_{2}^{E}$ are described as follows:
\begin{align*}
X_{1\link2,e}^{E}&=\{X_{4y-1-4k,4y} : y\in [\PE{n/2}], k<y\}\cup\{X_{4x,4y} : x\in[\PE{n/4}], y\in[\PE{n/2}+1,n-x]\}\\
&=\{X_{i,j} : V_i\in V_{1,e}^{E},V_j\in V_{2,e}^{E},j\in 2\bZ+1,j>i\}\\
&\hspace{5cm}\cup\{X_{i,j} : V_i\in V_{1,e}^{E},V_j\in V_{2,e}^{E},j\in 2\bZ\cap [4n-i]\}\eqsp,\\
X_{1\link2,o}^{E}&=\{X_{i,j} : V_i\in V_{1,o}^{E},V_j\in V_{2,o}^{E},j\in 2\bZ+1,j>i\}\\
&\hspace{5cm}\cup\{X_{i,j} : V_i\in V_{1,o}^{E},V_j\in V_{2,o}^{E},j\in 2\bZ\cap [4n-i]\}\eqsp.
\end{align*}
\end{enumerate}
By \eqref{eq:Opponents}, for any $q\in [2,\ordermax{}{E}]$, edges between $V_i$ and $V_j$ both in $V_{q}^{E}$ are:
\begin{align*}
\notag X_{q \link q,e}^{E}&=\{X_{i,j} : V_i\in V_{q,e}^{E},i\in 2\bZ+1, V_j\in V_{q,e}^{E},j\in 2\bZ\}\eqsp,\\
X_{q\link q,o}^{E}&=\{X_{i,j} : V_i\in V_{q,o}^{E},i\in 2\bZ+1, V_j\in V_{q,o}^{E},j\in 2\bZ\}\eqsp.
\end{align*}
Note that \eqref{eq:Opponents} shows also that there is no edge between $V_i\in V_{q,e}^{E}$ and $V_j\in V_{q,o}^{E}$. For all $2\le q\le \ordermax{}{E}$ and all $V_i\in V_{q}^{E}$ and $V_j\in V_{q+1}^{E}$, 
\begin{align*}
X_{q\link q+1,e}^{E}=&\{X_{i,j} : V_i\in V_{q,e}^{E},i\in(2\bZ+1),V_j\in V_{q+1,e}^{E},j\in 2\bZ\cap [i+4n-3]\}\\
&\hspace{3cm}\cup\{X_{i,j} : V_i\in V_{q,e}^{E},i\in 2\bZ,V_j\in V_{q+1,e}^{E},j\in 2\bZ+1\cap [i]\}\eqsp,\\
X_{q\link q+1,o}^{E}=&\{X_{i,j} : V_i\in V_{q,o}^{E},i\in(2\bZ+1),V_j\in V_{q+1,o}^{E},j\in2\bZ\cap [i+4n-3]\}\\
&\hspace{3cm}\cup\{X_{i,j} : V_i\in V_{q,o}^{E},i\in2\bZ,V_j\in V_{q+1,o}^{E},j\in( 2\bZ+1)\cap [i]\}\eqsp.
\end{align*}
Therefore, for all $2\le q\le \ordermax{}{E}$,
\begin{align*}
|X_{q\link q,e}^{E}|&=|\{i : V_i\in V_{q,e}^{E},i\in 2\bZ+1\}||\{j: V_j\in V_{q,e}^{E},j\in 2\bZ\}|\\
&=
\begin{cases}
 p^2&\text{ if }n-1=2p\eqsp,\\
 p(p+1)&\text{ if }n-1=2p+1\eqsp.
\end{cases}
\end{align*}
The same holds for $|X_{q\link q,o}^{E}|$ so that $|X_{q\link q}^{E}| = 2p^2$ if $n-1=2p$ and $|X_{q\link q}^{E}| = 2p(p+1)$ if $n-1=2p+1$. On the other hand,
\begin{align*}
|X_{q\link q+1,e}^{E}|=&\sum_{i : V_i\in V_{q,e}^{E},\,i\in(2\bZ+1)}|\{j : V_j\in V_{q+1,e}^{E}j\in 2\bZ\cap [i+4n-3]\}\\
&\hspace{3cm}+\sum_{i : V_i\in V_{q,e}^{E},\,i\in 2\bZ}|\{j: V_j\in V_{q+1,e}^{E},j\in 2\bZ+1\cap [i]\}|\\
&\hspace{-.37cm}=
\begin{cases}
2\sum_{i=1}^pi=p(p+1) &\text{ if }n-1=2p\eqsp,\\
\sum_{i=1}^pi+\sum_{i=1}^{p+1}i=(p+1)^2 &\text{ if }n-1=2p+1\eqsp.
\end{cases}
\end{align*}
As the same holds for $|X_{q\link q+1,o}^{E}|$, $|X_{q\link q+1}^{E}| = 2p(p+1)$ if $n-1=2p$ and $|X_{q\link q+1}^{E}| = 2(p+1)^2$ if $n-1=2p+1$. The proof is completed by writing $|X_{q}^{E}| = |X_{q\link q+1}^{E}|  + |X_{q+1\link q+1}^{E}|$.
\end{proof}

\section{Probabilistic study of the graphical model}\label{sec:ProbTools}

This section analyses stochastic processes whose conditional dependences are encoded in the graphical model of Figure~\ref{fig:generic:graphicalmodel}. To ease applications of these general results to our problem, we focus on a restricted class of such stochastic processes.

Let $\nobs\in\bN\setminus\{0\}$, $\pi_V$ be a distribution on a measurable space $\Vset$ and $\Xset$ be a discrete space. Let $K_i$ denote non-negative functions defined on $\Xset\times\Vset^2$ such that all $K_i(.,v,w)$ are probability distributions on $\Xset$. Let $\bP_{\pi_V}$ be the distribution on $\Vset^{\nobs+1}\times\Xset^{\nobs}$ defined by: 
\begin{equation}
\label{eq:law:chain}
\bP_{\pi_V}\pa{V_{1:\nobs+1}\in A_{1:\nobs+1},X_{1:\nobs}}
=\int\prod_{i=1}^{\nobs+1}\un{A_i}(v_i)\prod_{i=1}^{\nobs+1}\pi_V(\rmd v_i)\prod_{i=1}^{\nobs}K_i(X_i,v_i,v_{i+1})\eqsp.
\end{equation}
The random variables $(V_i)_{i\in \{1,\ldots,\nobs+1\}}$ are i.i.d. taking values in $\Vset$ with common distribution $\pi_V$ and $(X_i)_{i\in \{1,\ldots\nobs\}}$ is a stochastic process taking values in a discrete set $\Xset$ such that $(X_i)_{i\in \{1,\ldots,\nobs\}}$ are independent conditionally on $V$ and
\[
\bP_{\pi_V}(X_i=x|V_{1:\nobs+1})=\bP_{\pi_V}(X_i=x|V_{i},V_{i+1}) =K_i(x,V_{i},V_{i+1}),\qquad \forall i\in \{1,\nobs\},\forall x\in \Xset\eqsp.
\]
Therefore, $\bP_{\pi_V}$ is a generic probability distribution with conditional dependences encoded by the graphical model of Figure~\ref{fig:generic:graphicalmodel}.
Assume that there exist $\nu_i>0$ such that
\begin{equation}\label{Hyp:MinK_i}
\nu_i\le K_i(x,v,w)\le 1,\qquad \forall x\in\Xset,\forall i\in\bZ,\forall v,w\in\Vset\eqsp. 
\end{equation}
For some results, the following assumption is required.
\begin{equation}\label{Hyp:Stat}
\forall i\in\{1,\ldots,\nobs\},\qquad K_i=K\eqsp. 
\end{equation}
Whenever Assumption \eqref{Hyp:Stat} holds, we shall denote by $\nu$ a real number such that 
\begin{equation*}
\nu\le K(x,v,w)\le 1,\qquad \forall x\in\Xset,\forall v,w\in\Vset\eqsp. 
\end{equation*}
Note that by \eqref{eq:law:chain}, the sequence $(V_{k+1},X_{k})_{k\ge 0}$ is a Markov chain with transition kernel on $\Vset\times\Xset$ such that: 
\begin{align*}
\bP_{\pi_V}\pa{V_{k+1}\in A,X_k\middle|V_{k},X_{k-1}}&=\int\un{A}(v_{k+1})\pi_V(\rmd v_{k+1})K_k(X_k,V_k,v_{k+1})\ge \nu_k \pi_V(A)\eqsp.
\end{align*}
This uniform minoration condition ensures that the joint Markov chain $(V_{k+1},X_{k})_{k\ge 0}$ is geometrically ergodic and admits the whole space $\Vset\times\Xset$ as small set. Note also that, as defined by \eqref{eq:law:chain}, $\bP_{\pi_V}$ is the law of this Markov chain started from stationarity, the stationary distribution on $\Vset\times\Xset$ being $(A,x_0)\mapsto \int \un{A}(v_1)\pi_V(\rmd v_1) \pi_V(\rmd v_0)\condlik(x_0,v_0,v_1)$.

Lemma~\ref{lem:minorizationGen} first shows that, conditionally on the observations, $V_1,\ldots,V_{\nobs}$ is a backward Markov chain admitting the all state space as small set.
\begin{lemma}
\label{lem:minorizationGen}
For any $q\ge 1$, conditionally on $X_{q:\nobs}$, $(V_{\nobs},\ldots,V_1)$ is a Markov chain. Its transition kernels $(K^{V|X}_{\pi_V,k,q})_{q\le k <\nobs}$ are such that, for all $q\le k < \nobs$, there exists a measure $\mu_{k,q}$ satisfying for all measurable set $A$:
\begin{align*}
K^{V|X}_{\pi_V,k,q}(V_{k+1},A) = \bP_{\pi_V}\pa{V_{k}\in A\middle|V_{k+1:\nobs},X_{q:\nobs}}  &= \bP_{\pi_V}\pa{V_{k}\in A\middle|V_{k+1},X_{q:\nobs}}\ge \nu_k \mu_{k,q}(A)\eqsp.
\end{align*}
On the other hand, for all $1\le k <q$,
\[
K^{V|X}_{\pi_V,k,q}(V_{k+1},A) = \bP_{\pi_V}\pa{V_{k}\in A\middle|V_{k+1:\nobs},X_{q:\nobs}} = \pi_V(A)\eqsp.
\]
\end{lemma}
\begin{proof}
The Markov property is immediate. The case $1\le k <q$ follows from the independence of $V_k$ and $(V_{k+1:\nobs},X_{q:\nobs})$. Then, for any $q\le k < \nobs$ and all measurable set $A$,
\begin{align*}
\bP_{\pi_V}\pa{V_{k}\in A\middle|V_{k+1:\nobs},X_{q:\nobs}} &= \bP_{\pi_V}\pa{V_{k}\in A\middle|V_{k+1},X_{q:k}}\\
&= \frac{\int \1_{A}(v_{k})\pi_V(\rmd v_{k})K_k(X_{k},v_{k},V_{k+1})\bP_{\pi_V}(X_{q:k-1}|v_{k})}{\int \pi_V(\rmd v_{k})K_k(X_{k},v_{k},V_{k+1})\bP_{\pi_V}(X_{q:k-1}|v_{k})}\eqsp,
\end{align*}
with the conventions $\bP_{\pi_V}(X_{q:q-1}|V_{q})=1$. By Assumption \ref{assum:strongmix},
\[
\bP_{\pi_V}\pa{V_{k}\in A\middle|V_{k+1},X_{q:\nobs}}
\ge \nu_k\frac{\int \1_{A}(v_{k})\pi_V(\rmd v_{k})\bP_{\pi_V}(X_{q:k-1}|v_{k})}{\int \pi_V(\rmd v_{k})\bP_{\pi_V}(X_{q:k-1}|v_{k})}\eqsp. 
\]
The proof is then completed by choosing:
\[
\mu_{k,q}(A) = \frac{\int \1_{A}(v_{k})\pi_V(\rmd v_{k})\bP_{\pi_V}(X_{q:k-1}|v_{k})}{\int \pi_V(\rmd v_{k})\bP_{\pi_V}(X_{q:k-1}|v_{k})}\eqsp. 
\]
\end{proof}

\noindent
Lemma~\ref{lem:totalvariationGen} shows the contraction properties of the Markov kernel of the chain $V$ conditionally on the observations. It is a direct consequence of the minoration condition given in Lemma~\ref{lem:minorizationGen}, see for instance \cite[Sections III.9 to III.11]{lindvall:1992} or \cite[Corollary~4.3.9 and Lemma~4.3.13]{cappe:moulines:ryden:2005}. Let $\|\cdot\|_{\mathsf{tv}}$ be the total variation norm defined, for any measurable set $(\mathsf{Z},\mathcal{Z})$ and any finite signed measure $\xi$ on $(\mathsf{Z},\mathcal{Z})$, by
\[
\|\xi\|_{\mathsf{tv}} = \sup\left\{\int f(z) \xi(\rmd z) \;;\;f\; \mbox{measurable real function on\;}\mathsf{Z}\;\mbox{such that\;}\|f\|_{\infty} = 1\right\}\eqsp.
\]
\begin{lemma}
\label{lem:totalvariationGen}
For all measures $\mu_1$, $\mu_2$ and all $1\le q\le k < \nobs$,
\[
\left\|\int \mu_1 (\rmd x)K^{V|X}_{\pi_V,k,q}(x,\cdot) - \int \mu_2 (\rmd x)K^{V|X}_{\pi_V,k,q}(x,\cdot) \right\|_{\mathsf{tv}}\le \left(1-\nu_k\right)\|\mu_1-\mu_2\|_{\mathsf{tv}}\le \left(1-\nu_k\right)\eqsp.
\]
In particular, by induction,
\begin{equation}
\label{eq:exp:forgettingGen}
\left\|\int \left\{\mu_1(\rmd v_{\nobs})-\mu_2(\rmd v_{\nobs})\right\} K^{V|X}_{\pi_V,\nobs-1,q}(v_{\nobs},\rmd v_{\nobs-1})\ldots K^{V|X}_{\pi_V,k,q}(v_{k+1},\cdot)\right\|_{\mathsf{tv}}\le \prod_{i=k}^{\nobs-1}\left(1-\nu_i\right)\eqsp.
\end{equation}
\end{lemma}

\noindent
Lemma~\ref{lem:exp:forgetGen} proves a key loss of memory property of the backward chain $X_q$, with geometric rate of convergence. Whenever it is necessary, we adopt the convention  $\prod_{k=\ell}^{m}a_k = 1$ for any $(a_{\ell},\ldots,a_m)$ and any $\ell>m$.
\begin{lemma}
\label{lem:exp:forgetGen}
For any $1\le q\le \nobs-1$,
\begin{align}
\label{eq:bound:loglikGen}
\left|\log \bP_{\pi_V}\left(X_{q}\middle|X_{q+1:\nobs}\right)\right|&\le \log \left(\nu_q^{-1}\right)\eqsp. 
\end{align}
For all $\ell\ge1$,  $1\le q\le \nobs-1$,
\begin{equation}\label{eq:bound:forgetGen}
\absj{\log \bP_{\pi_V}\pa{X_{q}\middle|X_{q+1:\nobs}} - \log \bP_{\pi_V}\pa{X_{q}\middle|X_{q+1:\nobs+\ell}}}
\leq \nu^{-1}_q\prod_{k=q+1}^{\nobs-1}(1-\nu_k) \eqsp.
\end{equation}
\end{lemma}
\begin{proof}
To prove \eqref{eq:bound:forgetGen}, for $1\le q< \nobs$, note that by Lemma~\ref{lem:minorizationGen},  
\begin{equation}
\label{eq:decomp:prod:kernelGen}
\bP_{\pi_V}\pa{X_{q}\middle|X_{q+1:\nobs}} = \int \bP_{\pi_V}\pa{\rmd v_{\nobs}\middle | X_{q+1:\nobs}} \pa{\prod_{k =q+1}^{\nobs-1}K^{V|X}_{\pi_V,k,q+1}(v_{k+1},\rmd v_k)}\pi_V(\rmd v_q)K_{q}(X_q,v_q,v_{q+1})\eqsp.
\end{equation}
Likewise,
\begin{multline}
\label{eq:decomp:prod:kernel:plusGen}
\bP_{\pi_V}\pa{X_{q}\middle|X_{q+1:\nobs+\ell}} \\
= \int \bP_{\pi_V}\pa{\rmd v_{\nobs}\middle | X_{q+1:\nobs+\ell}} \pa{\prod_{k =q+1}^{\nobs-1}K^{V|X}_{\pi_V,k,q+1}(v_{k+1},\rmd v_k)}\pi_V(\rmd v_q)K_{q}(X_q,v_q,v_{q+1})\eqsp.
\end{multline}
Then, by Lemma~\ref{lem:minorizationGen} and \eqref{eq:exp:forgettingGen}, combining \eqref{eq:decomp:prod:kernelGen} and \eqref{eq:decomp:prod:kernel:plusGen} yields:
\begin{multline*}
\absj{\bP_{\pi_V}\pa{X_{q}\middle|X_{q+1:\nobs+\ell}}-  \bP_{\pi_V}\pa{X_{q}\middle|X_{q+1:\nobs}}}\\
\leq \pa{\prod_{k=q+1}^{\nobs-1}(1-\nu_k)} \sup_{v_{q+1}\in\Vset}\absj{\int\pi_V(\rmd v_q)K_{q}(X_q,v_q,v_{q+1})}
\le\prod_{k=q+1}^{\nobs-1}(1-\nu_k)\eqsp.
\end{multline*}
\eqref{eq:bound:forgetGen} is then a direct consequence of \eqref{eq:decomp:prod:kernelGen}, \eqref{eq:decomp:prod:kernel:plusGen} and the fact that 
for all $x,y>0$, $|\log x - \log y| \le |x-y|/x\wedge y$.  Inequality \eqref{eq:bound:loglikGen} follows from \eqref{eq:decomp:prod:kernelGen}. 
\end{proof}

\noindent
Lemma~\ref{lem:IncrementsGen} is the crucial result to bound the increments of the log-likelihood.

\begin{lemma}\label{lem:IncrementsGen}
For all distributions $\pi_V,\pi'_V\in\Pi\cup\{\pi^{\star}\}$ and any $1\le q\le \nobs$,
\begin{multline*}
\absj{\log \bP_{\pi_V}(X_{q}|X_{q+1:\nobs}) - \log \bP_{\pi'_V}(X_{q}|X_{q+1:\nobs})} \\
\leq 2\sum_{\ell= 0}^{\nobs+1-q}(\nu_{q}\nu_{q+\ell-1}\nu_{q+\ell})^{-1}\pa{\prod_{k = q+1}^{q+\ell-1}(1-\nu_k)}\|\pi_V-\pi'_V\|_{\mathsf{tv}}\eqsp.
\end{multline*}
\end{lemma}
\begin{proof}
When $q=\nobs$,
\[
\bP_{\pi_V}(X_{\nobs}) - \bP_{\pi'_V}(X_{\nobs}) = \int\left\{\pi'^{\otimes 2}_V(\rmd v_{\nobs:\nobs+1})-\pi^{\otimes 2}_V(\rmd v_{\nobs:\nobs+1})\right\}K_{\nobs}(X_{\nobs},v_{\nobs},v_{\nobs+1})\eqsp.
\]
Thus $|\bP_{\pi_V}(X_{\nobs}) - \bP_{\pi'_V}(X_{\nobs})|\le 2 \|\pi_V-\pi'_V\|_{\mathsf{tv}}$.
When $1\le q\le \nobs-1$,
\begin{equation*}
\bP_{\pi_V}(X_{q}|X_{q+1:\nobs}) - \bP_{\pi'_V}(X_{q}|X_{q+1:\nobs}) 
= \sum_{\ell= 0}^{\nobs+1-q}\left\{\bP_{\ell}(X_{q}|X_{q+1:\nobs}) - \bP_{\ell+1}(X_{q}|X_{q+1:\nobs})\right\}\eqsp,
\end{equation*}
where $\bP_{\ell}$ is the joint distribution of $(X_{q:\nobs},V_{q:\nobs+1})$ when $(V_q,\ldots,V_{q+\ell-1})$ are i.i.d. $\pi'_V$ and $(V_{q+\ell},\ldots,V_{\nobs+1})$ are i.i.d. $\pi_V$. The first term in the telescopic sum is given by:
\begin{multline*}
\bP_{0}(X_{q}|X_{q+1:\nobs}) - \bP_{1}(X_{q}|X_{q+1:\nobs}) = \int \bP_{0}\left(\rmd v_{q+1}\middle | X_{q+1:\nobs}\right)\int\pi'_V(\rmd v_q)K_{q}(X_q,v_q,v_{q+1})\\
- \int \bP_{0}\left(\rmd v_{q+1}\middle | X_{q+1:\nobs}\right)\int\pi_V(\rmd v_q)K_{q}(X_q,v_q,v_{q+1})\eqsp,
\end{multline*}
where $\bP_{0}\left(V_{q+1}\middle | X_{q+1:\nobs}\right)$ is the distribution of $V_{q+1}$ conditionally on $X_{q+1:\nobs}$ when $(V_{q},\ldots,V_{\nobs+1})$ are i.i.d. $\pi_V$. As $V_q$ is independent of $(V_{q+1},X_{q+1:\nobs})$, this distribution is the same as the distribution of $V_{q+1}$ conditionally on $X_{q+1:\nobs}$ when $V_q\sim \pi'_V$ and $(V_{q+1},\ldots,V_{\nobs+1})$ are i.i.d. $\pi_V$. 
\[
\absj{\bP_{0}(X_{q}|X_{q+1:\nobs}) - \bP_{1}(X_{q}|X_{q+1:\nobs})}\le\|\pi_V-\pi'_V\|_{\mathsf{tv}}\eqsp.
\]
Then, for all $1\le \ell \le \nobs+2-q$,
\begin{equation*}
\bP_{\ell}\left(X_{q}\middle|X_{q+1:\nobs}\right) = \int \bP_{\ell}\left(\rmd v_{q+\ell}\middle | X_{q+1:\nobs}\right) \pa{\prod_{k = q+1}^{q+\ell-1}K^{V|X}_{\pi'_V,k,q+1}(v_{k+1},\rmd v_k)}\int\pi'_V(\rmd v_q)K_{q}(X_q,v_q,v_{q+1})\eqsp.
\end{equation*}
Therefore, by \eqref{eq:exp:forgettingGen},
\begin{multline*}
\left|\bP_{\ell}\left(X_{q}\middle|X_{q+1:\nobs}\right)-\bP_{\ell+1}\left(X_{q}\middle|X_{q+1:\nobs}\right)\right|\\
\le \left(\prod_{k = q+1}^{q+\ell-1}(1-\nu_k)\right)\norm{\bP_{\ell}\left(V_{q+\ell}\middle | X_{q+1:\nobs}\right)-\bP_{\ell+1}\left(V_{q+\ell}\middle | X_{q+1:\nobs}\right)}_{\mathsf{tv}}\eqsp,
\end{multline*}
where $\bP_{\ell}\left(V_{q+\ell}\middle | X_{q+1:\nobs}\right)$ is the distribution of $V_{q+\ell}$ conditionally on $X_{q+1:\nobs}$ when $(V_q,\ldots,V_{q+\ell-1})$ are i.i.d. $\pi'_V$ and $(V_{q+\ell},\ldots,V_{\nobs+1})$ are i.i.d. $\pi_V$. It remains to show that 
\[
\norm{\bP_{\ell}\left(V_{q+\ell}\middle | X_{q+1:\nobs}\right)-\bP_{\ell+1}\left(V_{q+\ell}\middle | X_{q+1:\nobs}\right)}_{\mathsf{tv}}\le  2(\nu_{q}\nu_{q+\ell-1}\nu_{q+\ell})^{-1} \|\pi_V-\pi'_V\|_{\mathsf{tv}}
\]
which amounts to showing that for all $f$ such that $\norm{f}_\infty\le 1$,
\[
\left|\int f(v_{q+\ell})\left\{\bP_{\ell}\left(\rmd v_{q+\ell}\middle | X_{q+1:\nobs}\right)-\bP_{\ell+1}\left(\rmd v_{q+\ell}\middle | X_{q+1:\nobs}\right)\right\}\right|\le 2(\nu_{q}\nu_{q+\ell-1}\nu_{q+\ell})^{-1} \|\pi_V-\pi'_V\|_{\mathsf{tv}}\eqsp.
\]
Write, for all $1\le \ell \le \nobs+2-q$,
\begin{equation}
\label{eq:defL}
\mathrm{L}_{\ell}(\rmd v,X)= \prod_{m=q+1}^{q+\ell-1}\pi'_V(\rmd v_m)\prod_{m=q+\ell}^{\nobs+1}\pi_V(\rmd v_m)\prod_{m=q+1}^{\nobs}K_m(X_{m},v_{m},v_{m+1})\eqsp.
\end{equation}
We have
\[
\int f(v_{q+\ell})\bP_{\ell}\left(\rmd v_{q+\ell}\middle | X_{q+1:\nobs}\right)=\frac{\int f(v_{q+\ell})L_\ell(\rmd v,X)}{\int L_\ell(\rmd v,X)}\eqsp.
\]
Therefore,
\begin{align*}
 \int f(v_{q+\ell})\left\{\bP_{\ell}\left(\rmd v_{q+\ell}\middle | X_{q+1:\nobs}\right)-\bP_{\ell+1}\left(\rmd v_{q+\ell}\middle | X_{q+1:\nobs}\right)\right\} &\\
 &\hspace{-6cm}=\int f(v_{q+\ell})\pa{\frac{L_\ell(\rmd v,X)}{\int L_\ell(\rmd v,X)}-\frac{L_{\ell+1}(\rmd v,X)}{\int L_{\ell+1}(\rmd v,X)}}\eqsp,\\
 &\hspace{-6cm}= \int f(v_{q+\ell})\frac{L_\ell(\rmd v,X)-L_{\ell+1}(\rmd v,X)}{\int L_\ell(\rmd v,X)}\\
 &\hspace{-2cm}+\int f(v_{q+\ell})\frac{L_{\ell+1}(\rmd v,X)}{\int L_{\ell+1}(\rmd v,X)}\frac{\int \cro{L_{\ell+1}(\rmd v,X)-L_\ell( \rmd v,X)}}{\int L_\ell(\rmd v,X)}\eqsp.
\end{align*}
Thus,
\begin{multline}\label{eq:TV1Gen}
\left|\int f(v_{q+\ell})\left\{\bP_{\ell}\left(\rmd v_{q+\ell}\middle | X_{q+1:\nobs}\right)-\bP_{\ell+1}\left(\rmd v_{q+\ell}\middle | X_{q+1:\nobs}\right)\right\}\right|
\le 2\frac{|\int \{L_\ell(\rmd v, X)-L_{\ell+1}(\rmd v, X)\}|}{\int L_\ell(\rmd v, X)}\eqsp. 
\end{multline}
By \eqref{eq:defL}, $1\le \ell \le \nobs+1-q$,
\begin{multline*}
\left|\int\{L_\ell(\rmd v, X)-L_{\ell+1}(\rmd v, X)\}\right| \\
= \left|\int\prod_{m=q+1}^{q+\ell-1}\pi'_V(\rmd v_m)\left\{\pi_V(\rmd v_{q+\ell})-\pi'_V(\rmd v_{q+\ell})\right\}\prod_{m=q+\ell+1}^{\nobs+1}\pi_V(\rmd v_m)\prod_{m=q+1}^{\nobs}K_m(X_{m},v_{m},v_{m+1})\right|
\end{multline*}
As $K_{q+\ell-1}$ and $K_{q+\ell}$ are upper bounded by 1, 
\begin{multline*}
\left|\int\{L_\ell(\rmd v, X)-L_{\ell+1}(\rmd v, X)\}\right| \le \left(\int\prod_{m=q+1}^{q+\ell-1}\pi'_V(\rmd v_m)\prod_{m=q+1}^{q+\ell-2}K_m(X_{m},v_{m},v_{m+1})\right)\\
\times \norm{\pi_V-\pi'_V}_{\textrm{tv}}\left(\int\prod_{m=q+\ell+1}^{n+1}\pi_V(\rmd v_m)\prod_{m=q+\ell+1}^{n}K_m(X_{m},v_{m},v_{m+1})\right)
\eqsp.
\end{multline*}
Similarly, since $K_{q+\ell-1}$ and $K_{q+\ell}$ are respectively lower bounded by $\nu_{q+\ell-1}$ and $\nu_{q+\ell}$,
\begin{multline*}
\int L_\ell(\rmd v, X) \ge \left(\int\prod_{m=q+1}^{q+\ell-1}\pi'_V(\rmd v_m)\prod_{m=q+1}^{q+\ell-2}K_m(X_{m},v_{m},v_{m+1})\right)\\
\times \nu_{q+\ell-1}\nu_{q+\ell}\left(\int\prod_{m=q+\ell+1}^{n+1}\pi_V(\rmd v_m)\prod_{m=q+\ell+1}^{n}K_m(X_{m},v_{m},v_{m+1})\right) \eqsp.
\end{multline*}
Plugging these bounds in \eqref{eq:TV1Gen} yields, for  $1\le \ell \le \nobs+1-q$,
\begin{equation*}
\left|\int f(v_{q+\ell})\left\{\bP_{\ell}\left(\rmd v_{q+\ell}\middle | X_{q+1:\nobs}\right)-\bP_{\ell+1}\left(\rmd v_{q+\ell}\middle | X_{q+1:\nobs}\right)\right\}\right|\le 2\,(\nu_{q+\ell-1}\nu_{q+\ell})^{-1}\norm{\pi_V-\pi'_V}_{\textrm{tv}}\eqsp. 
\end{equation*}
The proof is completed using the fact that 
for all $x,y>0$, $|\log x - \log y| \le |x-y|/x\wedge y$. 
\end{proof}

\noindent
Lemma~\ref{lem:BoundedDifferenceGen} is a key ingredient to prove bounded difference properties for log-likelihood based processes.
\begin{lemma}\label{lem:BoundedDifferenceGen}
For all $1\le q\le \nobs$ and all $q\le \tilde{q}\le \nobs$, let $\widetilde{X}_{q:\nobs}^{\tilde{q}}$ be  such that $\widetilde{X}_{\tilde{q}}^{\tilde{q}}\in \Xset$ and $\widetilde{X}_k^{\tilde{q}} = X_k$ for all $q\le k\le \nobs$ such that $k\neq \tilde{q}$. 
For any $1\le q\le \tilde{q}\le \nobs$,
\[
\left|\log \bP_{\pi_V}(X_{q}|X_{q+1:\nobs}) - \log \bP_{\pi_V}(\widetilde{X}_{q}^{\tilde{q}}|\widetilde{X}_{q+1:\nobs}^{\tilde{q}})\right| \leq \nu_q^{-1}\prod_{k = q+1}^{\tilde{q}-1}(1-\nu_k)\eqsp.
\]
\end{lemma}

\begin{proof}
If $q = \tilde q = \nobs$, then 
\begin{align*}
\absj{\bP_{\pi_V}(X_{\nobs}) - \bP_{\pi_V}(\widetilde{X}_{n}^{n})} &=  \absj{\int \pi_V(\rmd v_{\nobs})\pi_V(\rmd v_{\nobs+1})\left\{K_{\nobs}(X_\nobs,v_\nobs,v_{\nobs+1})-K_{\nobs}(\widetilde{X}^{\nobs}_\nobs,v_\nobs,v_{\nobs+1})\right\}}\eqsp,\\
& \le 1-\nu_\nobs\le 1\eqsp.
\end{align*}
Assume now that $1\le q <\nobs$. When $\tilde q = q$,
\begin{multline*}
\bP_{\pi_V}\left(X_{q}\middle|X_{q+1:\nobs}\right) - \bP_{\pi_V}(\widetilde{X}^{q}_{q}|\widetilde{X}^{q}_{q+1:\nobs}) \\
= \int \bP_{\pi_V}\left(\rmd v_{q+1}\middle | \widetilde{X}^{q}_{q+1:\nobs}\right) \pi_V(\rmd v_q)\left\{K_q(X_q,v_q,v_{q+1})-K_q(\widetilde{X}^{q}_q,v_q,v_{q+1})\right\}\eqsp,
\end{multline*}
which ensures that $|\bP_{\pi_V}(X_{q}|X_{q+1:\nobs}) - \bP_{\pi_V}(\widetilde{X}^{q}_{q}|\widetilde{X}^{q}_{q+1:\nobs})|\le 1- \nu_q\le 1$. When $\tilde q \ge q+1$, as for all $q+1\le k\le \tilde{q}-1$ the Markov transition kernel $K^{V|X}_{\pi_V,k,q+1}$ depends only on $\pi_V$, $K_k$ and $X_{q+1:k}$, 
\[
\bP_{\pi_V}\left(\widetilde{X}^{\tilde{q}}_{q}\middle|\widetilde{X}^{\tilde{q}}_{q+1:\nobs}\right) = \int \bP_{\pi_V}\left(\rmd v_{\tilde{q}}\middle | \widetilde{X}^{\tilde{q}}_{q+1:\nobs}\right) \pa{\prod_{k = q+1}^{\tilde{q}-1}K^{V|X}_{\pi_V,k,q+1}(v_{k+1},\rmd v_k)} \pi_V(\rmd v_q)K_q(X_q,v_q,v_{q+1})\eqsp.
\]
By Lemma~\ref{lem:totalvariationGen}, it follows that
\begin{multline*}
\left|\bP_{\pi_V}\left(X_{q}\middle|X_{q+1:\nobs}\right)-\bP_{\pi_V}\left(\widetilde{X}^{\tilde{q}}_{q}\middle|\widetilde{X}^{\tilde{q}}_{q+1:\nobs}\right)\right| \\
\le \left(\prod_{k = q+1}^{\tilde{q}-1}(1-\nu_k)\right) \sup_{v_{q+1}\in\Vset}\left|\int\pi_V(\rmd v_q)K_q(X_q,v_q,v_{q+1})\right|\eqsp.
\end{multline*}
The proof is completed using the fact that 
for all $x,y>0$, $|\log x - \log y| \le |x-y|/x\wedge y$. 
\end{proof}

\noindent
Let $\pi^*_V$ denote a probability distribution on $\Vset$ and let
\[
Z_{\pi_V}(X_{1:\nobs})=\frac1{\nobs}\sum_{q=1}^{\nobs}\cro{\log \bP_{\pi_V}(X_{q}|X_{q+1:\nobs})-\E_{\pi_V^*}\cro{\log \bP_{\pi_V}(X_{q}|X_{q+1:\nobs})}}\eqsp.
\]

\noindent
Lemma~\ref{lem:FGen} shows the concentration of $Z_{\pi_V}(X_{1:\nobs})$ around its expectation.
\begin{lemma}
\label{lem:FGen}
Assume that $K_i=K$ for all $i\in\bZ$, let $\cP$ denote a class of probability distributions on $\Vset$. There exists $c>0$ such that for all $t>0$,
\[
\bP_{\pi^*_V}\left(\left|\sup_{\pi_V\in\cP}\{Z_{\pi_V}(X_{1:\nobs})\}-\E_{\pi^*_V}\cro{\sup_{\pi_V\in\cP}\{Z_{\pi_V}(X_{1:\nobs})\}}\right|\ge c\nu^{-2}\frac{t}{\sqrt{\nobs}}\right)\le 2e^{-t^2}\eqsp.
\]
\end{lemma}
\begin{proof}
The proof relies on the bounded difference inequality for Markov chains \cite[Theorem 0.2]{dedecker:gouezel:2015}. 
To apply this result, $\sup_{\pi_V\in\cP}\{Z_{\pi_V}(X_{1:\nobs})\}$ has to be separately bounded. For all $1\le q\le\nobs$ and all $q\le \tilde{q}\le \nobs$, let $\widetilde{X}^{\tilde{q}}_{1:\nobs}$ such that $\widetilde{X}_{\tilde{q}}^{\tilde{q}}\in \Xset$ and $\widetilde{X}_k^{\tilde{q}} = X_k$ for all $1\le k\le \nobs$ such that $k\neq \tilde{q}$. 
Then, 
\begin{align*}
 |\sup_{\pi_V\in\cP}\set{Z_{\pi_V}(X_{1:\nobs})}-&\sup_{\pi_V\in\cP}\{Z_{\pi_V}(\widetilde{X}^{\tilde{q}}_{1:\nobs})\}|\\
 &\le\sup_{\pi_V\in\cP}\absj{\frac1{\nobs}\sum_{q=1}^{\nobs}\cro{\log \bP_{\pi_V}(X_{q}|X_{q+1:\nobs})-\log \bP_{\pi_V}(\widetilde{X}^{\tilde{q}}_{q}|\widetilde{X}^{\tilde{q}}_{q+1:\nobs})}}\\
 &\le\sup_{\pi_V\in\cP}\absj{\frac1{\nobs}\sum_{q=1}^{\tilde{q}}\cro{\log \bP_{\pi_V}(X_{q}|X_{q+1:\nobs})-\log \bP_{\pi_V}(\widetilde{X}^{\tilde{q}}_{q}|\widetilde{X}^{\tilde{q}}_{q+1:\nobs})}}\eqsp.
\end{align*}
By Lemma~\ref{lem:BoundedDifferenceGen}, for any distribution $\pi_V\in\cP$ and any $1\le q\le \nobs$,
\[
\absj{\frac1{\nobs}\sum_{q=1}^{\nobs}\cro{\log \bP_{\pi_V}(X_{q}|X_{q+1:\nobs})-\log \bP_{\pi_V}(\widetilde{X}^{\tilde{q}}_{q}|\widetilde{X}^{\tilde{q}}_{q+1:\nobs})}} \leq \frac1{\nobs}\sum_{q=1}^{\tilde{q}}\nu^{-1}(1-\nu)^{\tilde{q}-q-1}\eqsp.
\]
Hence, there exists $c>0$ such that,
\[
 |\sup_{\pi_V\in\cP}\set{Z_{\pi_V}(X_{1:\nobs})}-\sup_{\pi_V\in\cP}\{Z_{\pi_V}(\widetilde{X}^{\tilde{q}}_{1:\nobs})\}|\le \frac{c}{\nu^2\nobs}\eqsp.
\]
The proof is concluded by \cite[Theorem 0.2]{dedecker:gouezel:2015}.
\end{proof}

\noindent
Lemma~\ref{lem:inrcZGen} shows the subgaussian concentration inequality of the increments of $Z_{\pi_V}(X_{1:\nobs})$.
\begin{lemma}
\label{lem:inrcZGen}
Assume that $K_i=K$ for all $i\in\bZ$, let $\pi_V$, $\pi'_V$ denote two probability distributions on $\Vset$. 
 Then, there exists $c>0$ such that for all $n\ge 1$, $t>0$,
\begin{equation}\label{eq:ConcIncrementsGen}
 \bP_{\pi_V^*}\pa{\absj{\sqrt{\nobs}\left\{Z_{\pi_V}(X_{1:\nobs}) -Z_{\pi'_V}(X_{1:\nobs})\right\}}>t}\le \exp\cro{-\frac{t^2}{\pa{c\nu^{-5}d(\pi,\pi')}^2}}\eqsp.
\end{equation}
\end{lemma}
\begin{proof}
To prove that the increments $Z_{\pi_V}-Z_{\pi'_V}$ are separately bounded, consider, for all $1\le \tilde{q}\le \nobs$,  $\widetilde{X}^{\tilde{q}}_{1:\nobs}$ such that $\widetilde{X}_{\tilde{q}}^{\tilde{q}}\in \Xset$ and $\widetilde{X}_k^{\tilde{q}} = X_k$ for all $1\le k\le \nobs$ such that $k\neq \tilde{q}$. Then, by Lemma~\ref{lem:BoundedDifferenceGen},
\begin{align*}
\absj{Z_{\pi_V}(X_{1:\nobs}) -Z_{\pi_V}(\widetilde{X}^{\tilde{q}}_{1:\nobs})}&= \absj{\frac1{\nobs}\sum_{q=1}^{\nobs}\cro{\log \bP_{\pi_V}(X_{q}|X_{q+1:\nobs})-\log \bP_{\pi_V}(\widetilde X^{\tilde q}_{q}|\widetilde{X}^{\tilde q}_{q+1:\nobs})}}\eqsp,\\
&\le \frac1{\nobs}\sum_{q=1}^{\tilde q}\absj{\log \bP_{\pi_V}(X_{q}|X_{q+1:\nobs})-\log \bP_{\pi_V}(\widetilde X^{\tilde q}_{q}|\widetilde{X}^{\tilde q}_{q+1:\nobs})}\eqsp.
\end{align*}
On one hand, by Lemma~\ref{lem:IncrementsGen},
\begin{align*}
\absj{\log \bP_{\pi_V}(X_{q}|X_{q+1:\nobs}) - \log \bP_{\pi'_V}(X_{q}|X_{q+1:\nobs})}\le  2\nu^{-4}\|\pi_V-\pi'_V\|_{\mathsf{tv}}\eqsp.
\end{align*}
On the other hand, by Lemma~\ref{lem:BoundedDifferenceGen}, for any $1\le q\le \tilde{q}\le\nobs$,
\[
\left|\log \bP_{\pi_V}(X_{q}|X_{q+1:\nobs}) - \log \bP_{\pi_V}(\widetilde{X}_{q}^{\widetilde{q}}|\widetilde{X}_{q+1:\nobs}^{\widetilde{q}})\right| \leq \nu^{-1}(1-\nu)^{\tilde{q}-q-1}\eqsp.
\]
Thus,
\begin{multline*}
 \absj{\pa{Z_{\pi_V}(X_{1:\nobs}) -Z_{\pi'_V}(X_{1:\nobs})}-\pa{Z_{\pi_V}(\widetilde{X}^{\tilde{q}}_{1:\nobs})-Z_{\pi'_V}(\widetilde{X}^{\tilde{q}}_{1:\nobs})}}\\
\le\frac{2\nu^{-4}}{\nobs}\sum_{q=1}^{\tilde{q}}\cro{\|\pi_V-\pi'_V\|_{\mathsf{tv}}\wedge(1-\nu)^{\tilde{q}-q}}\le \frac{2\nu^{-5}}{\nobs}d(\pi,\pi')\eqsp.
\end{multline*}
Eq~\eqref{eq:ConcIncrementsGen} follows by plugging these bounded differences properties in \cite[Theorem 0.2]{dedecker:gouezel:2015}.
\end{proof}

\section{Proofs of the main results}
\label{sec:MainProofs}
 When H\ref{assum:strongmix} holds and $E=E^{n,N}_{\text{RR}}$, $(V_{2: \ordermax{}{E}}^{E},X_{2:\ordermax{}{E}-1}^{E})$ satisfies the assumptions of Section~\ref{sec:ProbTools} with 
 \[
 \pi_V=\pi^{\otimes n-1},\qquad K_i(X_{i}^{E},V^{E}_i,V^{E}_{i+1})=\prod_{X_{i,j}\in X_{i}^{E}}\condlik(X_{i,j},V_i,V_j),\qquad \nu_i=\varepsilon^{|X_{i}^{E}|}\eqsp.
 \]
 Moreover, it is proved in Section~\ref{sec:round:robin} that $\absj{X^{E}_{q}} = n(n-1)$ for $2 \le q\le \ordermax{}{E}-1$, 
 which implies that 
\begin{equation}\label{eq:LBnu}
  \nu_i\ge\varepsilon^{n^2}\eqsp.
\end{equation}

Throughout the proofs, the following conventions are used. For all $0 \le k \le \ordermax{}{E}$,
\[
v^{E}_k\in \cV^{|V_k^{E}|},\qquad \pi(\rmd v^{E}_{k})=\prod_{i : V_i\in V_k^{E}}\pi(\rmd v_i)\eqsp.
\]

\subsection{Proof of Theorem~\ref{thm:ConvVrais}}\label{ProofThm2}
The first inequality is a direct conclusion of Lemma~\ref{lem:exp:forgetGen}. The proof of the second inequality follows the same lines. Then, the log-likelihood is decomposed as follows
\begin{align}
\notag\log \likelihood{E}{\pi}\left(X^{E}\right) &= \log \likelihood{E}{\pi}\left(X_{2:\ordermax{}{E}-1}^{E}\right) + \log \likelihood{E}{\pi}\left(X_{0}^{E},X_{1}^{E},X_{\ordermax{}{E}}^{E}\middle|X_{2:\ordermax{}{E}-1}^{E}\right)\eqsp,\\
\label{eq:DecVrais}&= \sum_{q=2}^{\ordermax{}{E}-1}\!  \log \likelihood{E}{\pi}\!\left(X_{q}^{E}\middle|X_{q+1:\ordermax{}{E}-1}^{E}\right) + \log \likelihood{E}{\pi}\!\left(Z^{E}\middle|X_{2:\ordermax{}{E}-1}^{E}\right) \eqsp.
\end{align}
Let us first bound from above the last term in \eqref{eq:DecVrais}. 
\begin{align*}
\likelihood{E}{\pi}\left(Z^{E}\middle|X_{2:\ordermax{}{E}-1}^{E}\right) &= \int \likelihood{E}{\pi}\left(Z^{E},\rmd v_{0:2}^{E},\rmd v_{\ordermax{}{E}:\ordermax{}{E}+1}^{E}\middle|X_{2:\ordermax{}{E}-1}^{E}\right)\eqsp,\\
&= \int \likelihood{E}{\pi}\left(\rmd v_{0:2}^{E},\rmd v_{\ordermax{}{E}:\ordermax{}{E}+1}^{E}\middle|X_{2:\ordermax{}{E}-1}^{E}\right)\left\{\prod_{X_{i,j}\in Z^{E}}\condlik(X_{i,j},v_{i},v_j)\right\}\eqsp,
\end{align*}
By Assumption~\ref{assum:strongmix}
\begin{equation}\label{Eq:BoundCoeffRestes}
\varepsilon^{3n^2}\le \likelihood{E}{\pi}\left(Z^{E}\middle|X_{2:\ordermax{}{E}-1}^{E}\right)  \le 1\eqsp. 
\end{equation}
In particular, the last term in \eqref{eq:DecVrais} is $O(1)$ when $N$ grows to infinity. On the other hand,taking the limit as $\ell\to\infty$ in Lemma~\ref{lem:exp:forgetGen} and recalling that $\nu_i\geqslant \varepsilon^{n^2}$, see \eqref{eq:LBnu}, for any $\pi\in \Pi$,
\begin{equation}
\label{eq:linfty}
\frac{1}{\ordermax{}{E}}\sum_{q=2}^{\ordermax{}{E}-1} \left| \log \likelihood{E}{\pi}\left(X_{q}^{E}\middle|X_{q+1:\ordermax{}{E}-1}^{E}\right) - \ell_{\pi}^{n}(\shift^q{\bf X^{n}})\right|\le 
\frac{1}{\ordermax{}{E}}\sum_{q=2}^{\ordermax{}{E}-1}\frac{(1-\varepsilon^{n^2})^{q_E-q-2}}{\varepsilon^{n^2}}\le\frac{\varepsilon^{- 3n^2}}{\ordermax{}{E}}
\eqsp.
\end{equation}
%
By \eqref{eq:bound:loglikGen}, $|\ell_{\pi}^{n}({\bf X^{n}})|\le n^2 \log(\varepsilon^{-1})$, thus $\ell_{\pi}^{n}$ is integrable. Therefore, the ergodic theorem \cite[Theorem~24.1]{billingsley:1995} can be applied to $\sum_{q=2}^{\ordermax{}{E}-1}\ell_{\pi}^{n}(\shift^q{\bf X^{n}})/\ordermax{}{E}$ and \eqref{eq:ApproxRisk} follows.

\subsection{$R_{\bayes}$ is the excess risk function}
The following result shows that $R^n_{\bayes}$ is a non-negative function.
\begin{proposition}
\label{prop:max:likelihood}
For all $\pi\in\Pi$ and all $n\ge 1$, $R^n_{\bayes}(\pi)\ge 0$.
\end{proposition}
\begin{proof}
Let $\pi\in\Pi$ and $n\ge 1$. By \eqref{eq:bound:loglik},
\[
\kullback^n_{\bayes}(\pi) = \E_{\bayes}\left[\lim_{N\to \infty}\log \likelihood{E}{\pi}(X_{2}^{E}|X_{3:\ordermax{}{E}-1}^{E})\right]\eqsp.
\]
By Lebesgue's bounded convergence theorem 
\begin{align*}
\kullback^n_{\bayes}(\pi) &= \lim_{N\to \infty}\E_{\bayes}\left[\log \likelihood{E}{\pi}(X_{2}^{E}|X_{3:\ordermax{}{E}-1}^{E})\right]\\
&= \lim_{N\to \infty}\E_{\bayes}\left[\E_{\bayes}\left[\log \likelihood{E}{\pi}(X_{2}^{E}|X_{3:\ordermax{}{E}-1}^{E})\middle|X_{3:\ordermax{}{E}-1}^{E}\right]\right]\eqsp.
\end{align*}
Therefore, 
\begin{align*}
R^n_{\bayes}(\pi) &=  \lim_{N\to \infty}\left\{\E_{\bayes}\left[\E_{\bayes}\left[\log \likelihood{E}{\bayes}(X_{2}^{E}|X_{3:\ordermax{}{E}-1}^{E}) - \log \likelihood{E}{\pi}(X_{2}^{E}|X_{3:\ordermax{}{E}-1}^{E})\middle|X_{3:\ordermax{}{E}-1}^{E}\right]\right]\right\}\eqsp,
\end{align*}
and the latter is non negative since the term in the expectation is a Kullback-Leibler divergence.
\end{proof}

\subsection{Proof of Theorem~\ref{th:risk}}
\label{sec:proofs:risk}
As that for any $\pi\in\Pi\cup\{\bayes\}$, $\Lo{E}{\pi} = \log \likelihood{E}{\pi}(X^{E})$, the excess loss satisfies:
\begin{align*}
R^{n}_{\bayes}(\MLE^{E}) =&\;\kullback^n_{\bayes}(\bayes)-\E_{\bayes}\left[\frac{1}{\ordermax{}{E}}\Lo{E}{\bayes}\right]+\E_{\bayes}\left[\frac{1}{\ordermax{}{E}}\Lo{E}{\bayes}\right]-\frac{1}{\ordermax{}{E}}\Lo{E}{\bayes}\\
&+\frac{1}{\ordermax{}{E}}\Lo{E}{\bayes}-\frac{1}{\ordermax{}{E}}\Lo{E}{\MLE^{E}}+\frac{1}{\ordermax{}{E}}\Lo{E}{\MLE^{E}}-\E_{\bayes}\left[\frac{1}{\ordermax{}{E}}\Lo{E}{\MLE^{E}}\right]\\
&+\E_{\bayes}\left[\frac{1}{\ordermax{}{E}}\Lo{E}{\MLE^{E}}\right]-\kullback^n_{\bayes}(\MLE^{E})\eqsp.
\end{align*}
By definition $\Lo{E}{\bayes}-\Lo{E}{\MLE^{E}}\le 0$. Thus,
\begin{align*}
R^{n}_{\bayes}(\MLE^{E}) 
&\le 2\;\sup_{\pi\in\Pi\cup\{\pi^*\}}\set{\left|\kullback^{\bayes}(\pi)-\frac{\E_{\bayes}\left[\Lo{E}{\pi}\right]}{\ordermax{}{E}}\right|
+\left|\frac{1}{\ordermax{}{E}}\E_{\bayes}\left[\Lo{E}{\pi}\right]-\frac{\Lo{E}{\pi}}{\ordermax{}{E}}\right|}\eqsp.
\end{align*}
For all $\pi\in\Pi$, as, for any $q\in \mathbb{Z}$, $\E_{\bayes}\left[\ell_{\pi}^{n}({\bf X^{n}})\right]=\E_{\bayes}\left[\ell_{\pi}^{n}(\shift^q{\bf X^{n}})\right]$,
\[
\kullback^{\bayes}(\pi)=\frac{1}{\ordermax{}{E}}\E_{\bayes}\left[\sum_{q=2}^{\ordermax{}{E}-1}\ell_{\pi}^{n}(\shift^q{\bf X^{n}})\right]+\frac{1}{\ordermax{}{E}}\E_{\bayes}\left[2\ell_{\pi}^{n}({\bf X^{n}})\right]\eqsp.
\]
Moreover, if $Z^{E} = X_{0}^{E}\cup X_{1}^{E}\cup X_{\ordermax{}{E}}^{E}$,
\[
\Lo{E}{\pi} = \log \likelihood{E}{\pi}(X^{E})=\sum_{q=2}^{\ordermax{}{E}-1} \log \likelihood{E}{\pi}\left( X_{q}^{E}\middle| X_{q+1:\ordermax{}{E}-1}^{E}\right)+\log \likelihood{E}{\pi}\left(Z^{E}\middle|X_{2:\ordermax{}{E}-1}^{E}\right)\eqsp.
\]
Therefore,
\begin{multline*}
\left|\kullback^{\bayes}(\pi)-\frac{\E_{\bayes}\left[\Lo{E}{\pi}\right]}{\ordermax{}{E}}\right| 
\le \frac{1}{\ordermax{}{E}} \E_{\bayes}\left[\sum_{q=2}^{\ordermax{}{E}-1}\left|\ell_{\pi}^{n}(\shift^q{\bf X^{n}})- \log \likelihood{E}{\pi}\left( X_{q}^{E}\middle| X_{q+1:\ordermax{}{E}-1}^{E}\right)\right|\right]\\
+ \frac{1}{\ordermax{}{E}} \E_{\bayes}\left[\left|2\ell_{\pi}^{n}({\bf X^{n}})\right|+ \left|\log \likelihood{E}{\pi}\left(Z^{E}\middle|X_{2:\ordermax{}{E}-1}^{E}\right)\right|\right]\eqsp.
\end{multline*}
Then, by \eqref{eq:linfty}, \eqref{eq:bound:loglikGen} and \eqref{Eq:BoundCoeffRestes} and the inequality $x\leq e^x$, there exists $c$ such that:
\[
\sup_{\pi\in\Pi\cup\{\pi^*\}}\left|\kullback^{\bayes}(\pi)-\frac{\E_{\bayes}\left[\Lo{E}{\pi}\right]}{\ordermax{}{E}}\right|\le \frac{c\varepsilon^{-3n^2}}{\ordermax{}{E}}\eqsp.
\]
This yields:
\begin{equation*}
 R^{n}_{\bayes}(\MLE^{E}) \le  \frac{c\varepsilon^{-3n^2}}{\ordermax{}{E}} + 2\,\sup_{\pi\in\Pi\cup\{\pi^*\}}\left|\frac{1}{\ordermax{}{E}}\E_{\bayes}\left[\Lo{E}{\pi}\right]-\frac{1}{\ordermax{}{E}}\Lo{E}{\pi}\right|\eqsp,
\end{equation*}
and therefore, by \eqref{Eq:BoundCoeffRestes},
\begin{equation}
\label{eq:risk:bound}
R^{n}_{\bayes}(\MLE^{E}) \le  \frac{c\varepsilon^{-3n^2}}{\ordermax{}{E}} + 2\,\sup_{\pi\in\Pi\cup\{\pi^*\}}\left|Z_{\pi_V}\right|\eqsp,
\end{equation}
where 
\[
Z_{\pi}=\frac1{\ordermax{}{E}}\sum_{q=2}^{\ordermax{}{E}-1}\cro{\log \likelihood{E}{\pi}(X^{E}_{q}|X^{E}_{q+1:\ordermax{}{E}})-\E_{\bayes}\cro{\log \likelihood{E}{\pi}(X^{E}_{q}|X^{E}_{q+1:\ordermax{}{E}})}}\eqsp.
\]
Lemma~\ref{lem:FGen} applies by assumption H\ref{assum:strongmix} since $E=E^{n,N}_{\text{RR}}$, therefore, there exists $c>0$ such that,
\begin{equation}\label{eq:ConcSup}
\bP_{\bayes}\pa{\absj{\sup_{\pi\in\Pi\cup\{\pi^*\}}Z_{\pi}-\E_{\bayes}\cro{\sup_{\pi\in\Pi\cup\{\pi^*\}}Z_{\pi}}}>c\varepsilon^{-2n^2}\frac{t}{\sqrt{\ordermax{}{E}}}}\le e^{-t^2},\qquad \forall t>0\eqsp.
\end{equation}
Furthermore, by Lemma~\ref{lem:inrcZGen}, the increments of $Z_\pi$ have subgaussian tails.
\[
\bP_{\bayes}\pa{\sqrt{\ordermax{}{E}}\absj{Z_{\pi}-Z_{\pi'}}>t}\le \exp\pa{-\frac{t^2}{\pa{c\varepsilon^{-5n^2}d(\pi^{\otimes |V_2^{E}|},(\pi')^{\otimes |V_2^{E}|})}^2}},\qquad \forall t>0\eqsp. 
\]
Now it is easy to check that
\[
\norm{\pi^{\otimes |V_2^{E}|}-(\pi')^{\otimes |V_2^{E}|}}_{{\sf tv}}\le |V_2^{E}|\norm{\pi-\pi'}_{{\sf tv}}\eqsp.
\]
Therefore, $d(\pi^{\otimes |V_2^{E}|},(\pi')^{\otimes |V_2^{E}|})\le cn^2d(\pi,\pi')\le c\varepsilon^{-n^2}d(\pi,\pi')$, thus
\begin{equation}\label{eq:subGaussIncr}
\bP_{\bayes}\pa{\sqrt{\ordermax{}{E}}\absj{Z_{\pi}-Z_{\pi'}}>t}\le \exp\pa{-\frac{t^2}{\pa{c\varepsilon^{-6n^2}d(\pi,\pi')}^2}},\qquad \forall t>0\eqsp. 
\end{equation}
Then, by Dudley's entropy bound, see \cite{MR0220340} or \cite[Proposition 2.1]{Talagrand:2014},
\begin{equation}\label{eq:Dudley}
\E_{\bayes}\cro{\sup_{\pi\in\Pi\cup\{\bayes\}}Z_{\pi}(X^{E})}\le \frac{ce^{-6n^2}}{\sqrt{\ordermax{}{E}}}\int_0^{+\infty}\sqrt{\log \mathsf{N}(\Pi\cup\{\bayes\},d,\epsilon)}\rmd\epsilon\eqsp. 
\end{equation}
Plugging \eqref{eq:ConcSup} and \eqref{eq:Dudley} into \eqref{eq:risk:bound} concludes the proof.

\bibliographystyle{plain}
\bibliography{bradleyterrybib}

\end{document}